\documentclass[12pt, leqno]{article}
\usepackage{amsmath, amsthm, enumitem, thmtools}
\usepackage{amsfonts}
\usepackage{amssymb}
\usepackage[usenames,dvipsnames]{color}
\usepackage{graphicx}
\usepackage{bbm}
\usepackage[toc]{appendix}
\usepackage{bm}
\usepackage{listings}
\usepackage{nicefrac} 
\usepackage{xfrac}    
\usepackage[utf8]{inputenc}
\usepackage[T1]{fontenc}
\usepackage{hyperref}
\usepackage{pgffor}
\usepackage{xcolor}
\usepackage{soul}
\usepackage{algpseudocode}
\usepackage{tikz}

\usepackage{color,calc}

\makeatletter
\definecolor{shade}{gray}{0.8}
        {
          \raggedright
        \setlength{\rightmargin}{\leftmargin}
        \setlength{\itemsep}{-12pt}
        \setlength{\parsep}{20pt}
        \begin{lrbox}{\@tempboxa}%
        \begin{minipage}{\linewidth-2\fboxsep}
        }%
        {
        \end{minipage}%
        \end{lrbox}%
        \fcolorbox{black}{shade}{\usebox{\@tempboxa}}\newline
        }%
\makeatother

\newtheorem{theorem}{Theorem}
\newtheorem{lemma}{Lemma}

\newcommand{\iu}{\mathrm{i}} 

\renewcommand{\eqref}[1]{\hyperref[#1]{(\ref*{#1})}}

\newcommand{\dd}{\mathrm{d}}

\newcommand{\ee}{\mathbf{e}}

\newtheorem{remark}{Remark}

\newcommand*{\pref}[1]{\hyperref[#1]{(\ref*{#1})}}
\newcommand*{\refpref}[2]{\hyperref[#2]{\ref*{#1}(\ref*{#2})}}

\definecolor{amethyst}{rgb}{0.6, 0.4, 0.8}
\definecolor{applegreen}{rgb}{0.55, 0.71, 0.0}
\definecolor{aqua}{rgb}{0.0, 1.0, 1.0}
\definecolor{asparagus}{rgb}{0.53, 0.66, 0.42}
\definecolor{amber(sae/ece)}{rgb}{1.0, 0.49, 0.0}
 	\definecolor{armygreen}{rgb}{0.29, 0.33, 0.13}
	\definecolor{shitbrown}{rgb}{0.43, 0.21, 0.1}
	\definecolor{brightpink}{rgb}{1.0, 0.0, 0.5}
	\definecolor{brightube}{rgb}{0.82, 0.62, 0.91}
	 	\definecolor{byzantine}{rgb}{0.74, 0.2, 0.64}
		\definecolor{chartreuse(web)}{rgb}{0.5, 1.0, 0.0}

\title{$\alpha$-stable L\'evy processes entering the half space or a slab}

\author{ Andreas E. Kyprianou\thanks{
Department of Statistics
University of Warwick
Coventry
CV4 7AL, UK. E-mail: \texttt{andreas.kyprianou@warwick.ac.uk}
}, \, Sonny Medina\thanks{
Department of Mathematical Sciences
University of Bath
Bath
 BA2 7AY, UK. E-mail: \texttt{samj20@bath.ac.uk}.} \,  and  \, Juan Carlos Pardo\thanks{CIMAT A. C., Calle Jalisco s/n, Col. Valenciana, A. P. 402, C.P. 36000, Guanajuato, Gto., Mexico. Email \texttt{jcpardo@cimat.mx}}
}

\begin{document}

\maketitle
\begin{abstract} 
Recently a series of publications, including e.g. \cite{DEEP1, DEEP2, DEEP3, watson, KJCbook}, considered a number of new fluctuation identities for $\alpha$-stable L\'evy processes in one and higher dimensions by appealing to underlying Lamperti-type path decompositions.  In the setting of $d$-dimensional isotropic processes, \cite{DEEP3} in particular, developed so called $n$-tuple laws for first entrance and exit of balls.  Fundamental to these works is the notion that the paths can be decomposed via generalised spherical polar coordinates revealing an underlying Markov Additive Process (MAP) for which a more advanced form of excursion theory (in the sense of \cite{maisonneuve}) can be exploited. 

Inspired by this approach, we give a different decomposition of the $d$-dimensional isotropic $\alpha$-stable L\'evy processes in terms of orthogonal coordinates. Accordingly we are able to develop a number of $n$-tuple laws for first entrance into a half-space bounded by an $\mathbb{R}^{d-1}$ hyperplane, expanding on existing results of \cite{byczkowski2009bessel, tamura2008formula}. This gives us the opportunity to numerically construct the law of first entry of the process into a slab of the form $(-1,1)\times \mathbb{R}^{d-1}$ using a `walk-on-half-spaces' Monte Carlo approach in the spirit of the `walk-on-spheres' Monte Carlo method given in \cite{kyprianou2018unbiased}.

\medskip

\medskip

\end{abstract}

\section{Introduction}

Let $X=(X_t, t \geq 0)$ be a $d$-dimensional  $\alpha$-stable  L\'evy process (henceforth referred to as just a {\it stable process}) with probabilities $\mathbb{P}: =(\mathbb{P}_x, x \in \mathbb{R}^d)$. This means that $X$ has c\`adl\`ag paths with stationary and independent increments as well as respecting a property of  self-similarity: There is an $\alpha>0$ such that, for $c>0,$ and $x \in {\mathbb{R}^d},$
under $\mathbb{P}_x$, the law of  $(cX_{c^{-\alpha}t}, t \geq 0)$ is equal to  $\mathbb{P}_{cx}$.
It turns out that stable L\'evy processes necessarily have the scaling index $\alpha\in (0,2]$. The case $\alpha=2$ pertains to a standard $d$-dimensional Brownian motion, thus has a continuous path. Similar results are possible  in the diffusive setting but we nonetheless restrict ourselves  to  the isotropic pure jump setting of $\alpha\in(0,2)$  in dimension $d\geq2$ in order to bring forward features of the proof which are dependent on the jump structure. 

\smallskip

 To be more precise, this means, for all orthogonal transformations $U:\mathbb{R}^d \mapsto \mathbb{R}^d$ and $x \in \mathbb{R}^d,$ 
\[
 \quad \textit{the law of} \quad (UX_t, t \geq 0) \textit{ under} \quad \mathbb{P}_x  \textit{ is equal to }  (X_t,t\geq 0) \textit{ under } \mathbb{P}_{Ux}.
 \]

The stable L\'evy  process has a the jump measure $\Pi$ that satisfies 
\begin{equation}
\Pi(B) = \frac{2^{\alpha} \Gamma((d+\alpha)/2)}{\pi^{d/2} |\Gamma(-\alpha/2)|} \int_B \frac{1}{|y|^{\alpha+d}} {\rm d} y, \quad B \subseteq \mathcal{B}(\mathbb{R}^d).
\label{Pi}
\end{equation}
{The constant in the definition of $\Pi(B)$ can be arbitrary, however, our choice corresponds to the  one that allows us to identify the characteristic exponent L\'evy process as}
\begin{equation} \label{charexponent}
\Psi(\theta)=-\frac{1}{t}\log \mathbb{E} ({\rm e}^{\iu\theta \cdot X_t}) =|\theta|^{\alpha}, \quad \theta \in \mathbb{R}^d,
\end{equation} 
where we write $\mathbb{P}$ in preference to $\mathbb{P}_0$; more precisely, the coefficient of $|
\theta|^\alpha$ is one.

\smallskip

We are interested in describing how the paths of $X = (X^{(1)},\cdots, X^{(d)})$ behave in relation to hyperplanes and the half-spaces that they bound. 
Because $X$ has stationary and independent increments and is isotropic, it suffices to focus on the hyperplane $\mathbb{H}_0: = \{z\in\mathbb{R}^d: z^{(1)}=0\}$ and the two open half spaces that it partitions, $\mathbb{H}_{0\downarrow}:=\{z \in \mathbb{R}^{d}: z^{(1)} <0 \}$ and $\mathbb{H}_{0\uparrow}:=\{z \in \mathbb{R}^{d}: z^{(1)} >0 \}$.

As such, we need to point out that a subset of $k \in \{1,2, \dots ,d\}$ coordinates of $X$, under $\mathbb{P}_x$, is in fact a $k$-dimensional isotropic stable  process. This follows by evaluating \eqref{charexponent} on the vector $\theta_k:=(\theta^{(1)},\theta^{(2)},\dots,\theta^{(k)},0,\dots,0)$ and noting that $$\Psi(\theta_k)=|\theta_k|^{\alpha}.$$ In particular $X^{(1)}$ is a symmetric real valued stable  process. Moreover, $X^{(2:d)}: = (X^{(2)}, \cdots, X^{(d)})$ is an isotropic stable process in dimension $d-1$. By using  isotropy, the projective properties of  $X$ onto lower dimensions, and recalling some path properties of symmetric real-valued stable processes, it is worth noting  the following facts, which will be important for our forthcoming computations. 

\begin{itemize}
   \item The isotropic stable process touches $\mathbb{H}_{0}$ (i.e. $X^{(1)}$ touches 0) in finite time with probability one when $\alpha \in (1,2)$ and with probability zero when $\alpha \in (0,1]$. Moreover, in the almost sure sense, when $\alpha = 1$, $\liminf_{t\to\infty}|X^{(1)}|=0$, whereas, when $\alpha\in(0,1)$, $\liminf_{t\to\infty}|X^{(1)}|>0$.
    
    \item Due to the oscillatory behaviour of $X^{(1)}$, i.e. $\limsup_{t\to \infty} X^{(1)}_t = -\liminf_{t\to \infty} X^{(1)}_t=+\infty$, for any $x\in \mathbb{R}^d$ the process $\mathbb{P}_x$-almost surely visits $\mathbb{H}_{0\downarrow}:=\{z \in \mathbb{R}^{d}: z^{(1)} <0 \}$ and $\mathbb{H}_{0\uparrow}:=\{z \in \mathbb{R}^{d}: z^{(1)} >0 \}$ the `lower' and `upper' half-spaces an infinite number of times.

    \item The process $X$ started in $x=(x^{(1)},\dots,x^{(d)})$ is regular for both the half spaces $(x^{(1)},\infty)\times\mathbb{R}^{d-1}$ and $(-\infty,x^{(1)})\times\mathbb{R}^{d-1}$. That is, for $\tau_{x^{(1)}}^{\wedge}:=\inf\{t>0: X^{(1)}>x^{(1)}\}$ and $\tau_{x^{(1)}}^{\vee}:=\inf\{t>0: X^{(1)}<x^{(1)}\}$, we necessarily have $\mathbb{P}_x(\tau_{x^{(1)}}^{\wedge}=0)=\mathbb{P}_x(\tau_{x^{(1)}}^{\vee}=0)=1$.

    \item Crossing a hyperplane does not occur by creeping. This means, for $r\neq x^{(1)}$, where $x^{(1)}$ is the first component of the point of issue, $x\in\mathbb{R}^d$, of $X$,  with $\tau_r^{\vee}=\inf\{t>0:X_t^{(1)}<r\}$ and $\tau_r^{\wedge}=\inf\{t>0:X_t^{(1)}>r\}$,  $$\mathbb{P}_x(X_{\tau_r^{\vee}}^{(1)} < r)=\mathbb{P}_x(X_{\tau_r^{\wedge}}^{(1)} > r)=1.$$
    
\end{itemize}

Our results, given in the next two sections, pertain to the first passage problem around the aforementioned stopping times $\tau_r^{\vee}=\inf\{t>0:X_t^{(1)}<r\}$ and $\tau_r^{\wedge}=\inf\{t>0:X_t^{(1)}>r\}$. With the above properties in mind we will develop a theory of excursions from a directional minimum/maximum, which are related to the excursions from a hyperplane studied by \cite{burdzy}. Our excursion process is `indexed' by the set of times at which the process reaches a new minimum/maximum in the direction of $X^{(1)}$. What is important however is that such excursions are now indexed by the value of $X^{(2:d)}$ at the time the minimum/maximum is reached.  

\smallskip

The results we will achieve for first entry into the half-space $\{x\in\mathbb{R}^d: x^{(1)}<0\}$ will give us some interesting insights into the much more difficult problem of understanding the first passage problem into the slab $S : = (-1,1)\times \mathbb{R}^{d-1}$, based on the stopping time 
\[
\tau_S  = \inf\{t>0: X_t \in S\}= \inf\{t>0: X^{(1)}_t \in (-1,1)\}.
\]
 That is, we are interested in the law of $X_{\tau_S }$. This problem was recently solved in one dimension by \cite{watson}, in other words the distribution of $X^{(1)}_{\tau_S }$ is already known. Understanding the distribution of the additional variables $X^{(2:d)}_{\tau_S }$ and their correlation to $X^{(1)}_{\tau_S }$ (which is specifically unlike the setting of Brownian motion where there is independence) is a much harder problem.

\smallskip

Although we have been unable to provide a closed form expression for the distribution of $X_{\tau_S}$, we will show how the law of $(X^{(1)}_{t}, t\leq \tau_S )$ and $X^{(2:d)}_{\tau_S }$ are correlated through a conditional dependency which relies on the $(d-1)$-dimensional Cauchy distribution. Moreover, we introduce a new Monte Carlo method, based on the method of `walks-on-spheres' (cf. \cite{kyprianou2018unbiased}), albeit now a method of `walks-on-half-spaces', which allows us to numerically compute the law of $X_{\tau_S}$.

\section{First crossing of a hyperplane}
Our first result considers the one-sided point of closest reach to the hyperplane $\mathbb{H}_0: = \{x\in\mathbb{R}^d : x^{(1)}=0\}$, before crossing it for the first time. Specifically, we are interested in $X_{G(\tau_0^{\vee})}$, where 
    $$G(\tau_r^{\vee})=\sup\{s<\tau_r^{\vee}:X^{(1)}_s=\inf_{u\leq \tau_r^{\vee}} X^{(1)}_u\},$$
and we recall  $\tau_0^{\vee}:=\inf\{s>0:X_{s}^{(1)}<0\}$.

\begin{theorem}[One-sided point of closest reach to $\mathbb{H}_0$ at first crossing] \label{thm:pcr}
Let $x,y\in \mathbb{H}_{0\uparrow}$,  then
\begin{equation} \label{PCR}
    \mathbb{P}_x(X_{G(\tau_0^{\vee})}\in {\rm d}y ) =  C_{\alpha,d}  \frac{|y^{(1)}-x^{(1)}|^{\alpha/2}}{|y^{(1)}|^{\alpha/2}}|x-y|^{-d}{\rm d}y \quad \text{ for } 0<y^{(1)}<x^{(1)}, 
\end{equation}\\
where $C_{\alpha,d}=\pi^{-(d/2+1)}\Gamma(d/2)\sin(\alpha\pi/2)$.
\end{theorem}

The concept of `point of closest reach' comes from recent work in \cite{DEEP3}, 
which considered the point on the path of an isotropic stable process which is radially closest to the origin. 
This concept was later used in \cite{tsogi} to condition the same process to continuously hit a patch on the unit sphere from outside. 

\begin{remark}
If we consider $x,y\in \mathbb{H}_{0\downarrow}$ such that $x^{(1)}<y^{(1)}<0$, replace $\tau_{0}^{\vee}$ by the stopping time $\tau_0^{\wedge}=\inf\{s>0:X_s^{(1)}>0\}$ and define $\overline{G}(\tau_0^{\wedge}):=\sup\{s<\tau_0^{\wedge}:X_s^{(1)}=\sup_{u<\tau_0^{\vee}}X_u^{(1)}\}$ (the point of closest reach from below of $\mathbb{H}_0$) by symmetry, the kernel in \eqref{PCR} is also the density of $X_{\overline{G}(\tau_0^{\wedge})}$.
\end{remark}

The use of an excursion framework to characterise directional minima, together with explicit expressions for functionals under the excursion measure, allow us to obtain joint laws of related quantities at first passage times in regards half-spaces. These so-called $n$\textit{-tuple} laws at first entry/exit of a half-space provide a non-trivial extension to the simple undershoot-overshoot law.

\begin{theorem}[Triple law at the first crossing of a hyperplane]\label{thm:triplelaw}
Fix $r\in\mathbb{R}$ and define for $w,x,y \in \mathbb{H}_{r\uparrow} $ and $z \in \mathbb{H}_{r\downarrow} $, the following density
\[
\Xi_x(w,y,z):=A_{\alpha,d}\frac{|x^{(1)}-w^{(1)}|^{\alpha/2}}{|x-w|^d}\frac{|w^{(1)}-y^{(1)}|^{\alpha/2}}{|w-y|^d}|y-z|^{-\alpha-d},
\] 
where 
\[
A_{\alpha,d}=\frac{2^{\alpha}\Gamma(d/2)^2\Gamma((d+\alpha)/2)}{\pi^{3d/2}\Gamma(\alpha/2)^2|\Gamma(-\alpha/2)|}.
\]
 Then, for $x^{(1)}>w^{(1)}>r,$ $y^{(1)}> w^{(1)}$ and $z^{(1)}<r$,
    \begin{equation}
    \mathbb{P}_x(X_{G(\tau_r^{\vee})}\in {\rm d}w,X_{\tau_r^{\vee}-}\in {\rm d}y,X_{\tau_r^{\vee}}\in {\rm d}z)=\Xi_{x}(w,y,z)\, {\rm d}w\, {\rm d}y\, {\rm d}z.
\label{triplelaw}
    \end{equation}
\end{theorem}

Once again, by symmetry, the previous result can be easily extended to the case where the point of issue is in $\mathbb{H}_{r\downarrow}$ and  upward crossing occurs at $\tau_r^{\wedge}$. 
\smallskip

Technically speaking, the next two Lemmas could be derived by marginalising \eqref{triplelaw} and, as such may be regarded as corollaries of Theorem \ref{thm:triplelaw}. However, it turns out they can be derived directly in a more straightforward way. We therefore state them as lemmas rather than corollaries. 

\begin{lemma}[Overshoot and undershoot at first crossing of a hyperplane] \label{cor:marginal}
Fix $r\in \mathbb{R}$ and define for $x,y \in \mathbb{H}_{r\uparrow}$ and $z \in \mathbb{H}_{r\downarrow}$ 
$$\Xi_x(\bullet,y,z):={B}_{\alpha,d} \bigg( |x-y|^{\alpha-d}\int_0^{\zeta_r^{\vee}(x,y)}(u+1)^{-d/2}u^{\alpha/2-1}{\rm d}u \bigg ) |z-y|^{-(\alpha+d)}$$where
\[
\zeta_r^{\vee}(x,y):=\frac{4(x^{(1)}-r)(y^{(1)}-r)}{|x-y|^2}
\quad\text{ and }\quad {B}_{\alpha,d}=
\frac{ \Gamma((d+\alpha)/2)\Gamma(d/2)}{\pi^{d}|\Gamma(\alpha/2)^2\Gamma(-\alpha/2)|}.
\] 
Then, for $y^{(1)}>r>z^{(1)}$
\begin{equation}
\mathbb{P}_x(X_{\tau_r^{\vee}-}\in {\rm d}y,X_{\tau_r^{\vee}}\in {\rm d}z)=\Xi_{x}(\bullet,y,z){\rm d}y{\rm d}z.
\label{doublelaw}
\end{equation}
\end{lemma}

Finally, in a second lemma below we give the law of just the overshoot distribution. However, we note that its proof again avoids performing a marginalisation of, this time,  the density \eqref{doublelaw}, and we use a third approach for its proof. 

\begin{lemma}[Overshoot at first crossing of a hyperplane]\label{bgrfe}
Fix $r\in \mathbb{R}$ and define for $x \in \mathbb{H}_{r\uparrow}$ and $z \in \mathbb{H}_{r\downarrow}$ 
$$\Xi_x(\bullet, \bullet,z):=C_{\alpha,d}|x-z|^{-d}\frac{|r-x^{(1)}|^{\alpha/2}}{|r-z^{(1)}|^{\alpha/2}},$$ where $C_{\alpha,d}=\pi^{-(d/2+1)}\Gamma(d/2)\sin(\alpha\pi/2)$. Then
\begin{equation} 
  \quad \mathbb{P}_x(X_{\tau^{\vee}_r}\in {\rm d}z)=\Xi_x(\bullet, \bullet,z)\dd z\quad z^{(1)}<r<x^{(1)}.
\end{equation}
\end{lemma}

\begin{remark}\rm
Lemma \ref{bgrfe} is known already from \cite{byczkowski2009bessel} and Lemma  \ref{cor:marginal} is also implicitly known from the same reference, albeit stated up to multiplication by the stable jump measure as the Green function for $X$ killed on exiting $\mathbb{H}_{r\uparrow}$. The approach taken by  \cite{byczkowski2009bessel} uses potential analysis rather than the fluctuation theory approach that we use here. 
\end{remark}

\section{Conditional law and numerics for first entry into a slab}
The objective of this section is to give a partial distributional decomposition of  the density of the first hit into the infinite slab $S=(-1,1)\times \mathbb{R}^{d-1}$ of $(X,\mathbb{P})$, together with a Monte Carlo method for numerically determining it. Pre-emptively assuming it exists, the density we are interested in is the bivariate function $H(x,y)$ such that for $x\in \mathbb{R}^d\setminus S$,
$$\mathbb{P}_x(X_{\tau_S}\in \dd y,\tau_S<\infty)
=H(x,y)\dd y, \quad y\in (-1,1)\times \mathbb{R}^{d-1}.$$

\smallskip

It is known that in the case $\alpha\in (1,2)$ the process $X$ will hit the slab with probability one because $X^{(1)}$, its first coordinate, will hit any interval   with probability one; cf. \cite{watson}. However, in the case $\alpha\in (0,1)$, we must necessarily work on the  event $\{\tau_S<\infty\}$.

\smallskip

Because stable processes always make first passage into a half space by a jump (i.e. no creeping)  we know that the coordinate $X^{(1)}$ must make a random number of crossings of $(-1,1)$ before entering it. To this end, without loss of generality, we can start with $x^{(1)}>1$ and define $\sigma_0=0$. Then for $k = 1,2,\cdots$,  iteratively define, 
\begin{align*}
    \sigma_k&=\begin{cases}
        \inf\{t>\sigma_{k-1}:X^{(1)}_t>-1\} & \text{if $k$ is even }\\
                \inf\{t>\sigma_{k-1}:X^{(1)}_t<1\} & \text{if $k$ is odd },\\
    \end{cases}
\end{align*}
and the random index $N$ is such that $\{N=k\}$ agrees with $\tau_S=\sigma_k$.
\smallskip

Our first result concerns the conditional distribution of  $X^{(2:d)}_{\tau_S}$ given $(X^{(1)}_t, t\leq \tau_S)$. More precisely, we are interested in the conditional distribution of $X^{(2:d)}_{\tau_S}$ given $(X^{(1)}_{\sigma_k}, k = 0,1,\cdots, N)$.
To this end, let us define $\mathtt{Cauchy}_{d-1}(0, \gamma)$, a $(d-1)$-dimensional Cauchy distribution centred at 0 with scaling $\gamma>0$. That is to say, a random variable on $\mathbb{R}^{d-1}$ with density
\[
\frac{\Gamma(d/2)}{\pi^{d/2}}\frac{\gamma }{(\gamma^2+ x^2)^{d/2}}, \qquad x\in \mathbb{R}^{d-1}.
\]

\begin{theorem}[Conditional law of first entry into a slab]\label{conditionalslab} We have
\[
\emph{\text{Law}}\left(X_{\tau_S}^{(2:d)}\bigg| (X^{(1)}_{t}, t\leq \tau_S)\right)\sim x^{(2:d)}+\mathtt{Cauchy}_{d-1}\left(0,\, \sum_{k=1}^{N}|X^{(1)}_{\sigma_k}-X^{(1)}_{\sigma_{k-1}}|\right).
\]
\end{theorem} 
\begin{proof}
  By Lemma \ref{bgrfe} in the case $d\geq2$,
    $$\mathbb{P}_x(X_{\tau^{\vee}_r}\in {\rm d}y)=C_{\alpha,d}|x-y|^{-d}\frac{|r-x^{(1)}|^{\alpha/2}}{|r-y^{(1)}|^{\alpha/2}}{\rm d}y \quad \text{ for } y^{(1)}<r<x^{(1)}$$
 where $C_{\alpha,d}=\pi^{-(d/2+1)}\Gamma(d/2)\sin(\alpha\pi/2)$.
Recalling that the same formula is true for the case $d=1$; cf. Corollary 3.5 in Chapter 3 of \cite{KJCbook} we have,
\[
\mathbb{P}_x(X^{(1)}_{\tau^{\vee}_r}\in {\rm d}y^{(1)})=C_{\alpha,1}|x^{(1)}-y^{(1)}|^{-1}\frac{|r-x^{(1)}|^{\alpha/2}}{|r-y^{(1)}|^{\alpha/2}}{\rm d}y^{(1)} \quad \text{ for } y^{(1)}<r<x^{(1)}.
\]

 Thus for $y^{(2:d)}\in \mathbb{R}^{d-1}$,
\begin{equation}
\mathbb{P}_x(X_{\tau_r^{\vee}}^{(2:d)}\in \dd y^{(2:d)}|X_{\tau_r^{\vee}}^{(1)} =y^{(1)})=\frac{\Gamma(d/2)}{\pi^{d/2}}\frac{|x^{(1)}-y^{(1)}|}{\left( (x^{(1)}-y^{(1)})^2+ |x^{(2:d)}-y^{(2:d)}|^2 \right)^{d/2}}\, \dd y^{(2:d)}.
\label{ReferenceToEqnInTheproofOfThm3}
\end{equation}
 That is to say, under $\mathbb{P}_x$,
\begin{align*}
\text{Law}\left(X^{(2:d)}_{\tau_r^{\vee}}-x^{(2:d)}\bigg| X^{(1)}_{\tau_r^{\vee}}=y^{(1)}\right)&\sim 
\mathtt{Cauchy}_{d-1}\left( 0,\,|x^{(1)}-y^{(1)}|\right).\notag\\
\end{align*}
By symmetry we can show that for  $x^{(1)}<r<y^{(1)}$, Law$(X^{(2)}_{\tau_r^{\wedge}}|X^{(1)}_{\tau_r^{\wedge}}=y^{(1)})$ under $\mathbb{P}_x$, follows a centered Cauchy distribution with scale parameter $|y^{(1)}-x^{(1)}|$.  
 
\smallskip

Since we know that 
\[
\mathtt{Cauchy}_{d-1}(0,\gamma) = \gamma \times \mathtt{Cauchy}_{d-1}(0,1),
\] 
we can thus identify, for $k = 1, \cdots, N$, with $\sigma_0=0$, and hence $x^{(1)} = X_{\sigma_0}$,
\[
\text{Law}\left(X^{(2:d)}_{\sigma_k}-x^{(2:d)}| X^{(1)}_{\sigma_j},  \, j = 0,1,\cdots, k\right)\sim 
\sum_{i = 1}^k |X^{(1)}_{\sigma_i}-X^{(1)}_{\sigma_{i-1}}|\times\mathtt{Cauchy}^{(i)}_{d-1}(0,\,1),
\]
where $\mathtt{Cauchy}^{(i)}_{d-1}(0,\,1)$ are iid copies of $\mathtt{Cauchy}_{d-1}(0,\,1)$ and the result follows by infinite divisibility and scaling of the Cauchy distribution.
\end{proof}

The proof of Theorem \ref{conditionalslab} inspires the following numerical approach to simulating the law of $X_{\tau_S}$ using an approach which is reminiscent of the so-called `walk-on-spheres' Monte Carlo approach described for stable processes in \cite{kyprianou2018unbiased}. As such we refer to it as the `walk-on-half-spaces' Monte Carlo approach.
\bigskip

\hrule
\bigskip

\noindent{\bf Algorithm 1} (simulation of $X_{\tau_S}$ under $\mathbb{P}_x$ for $x\in\mathbb{R}^d$ with $x^{(1)}>1$, $\alpha\in[1,2)$){\bf .}
\bigskip

\hrule
\bigskip

\noindent\textbf{sample}  $x_1$ from the distribution of $X_{\tau_1^{\vee}}$ under $\mathbb{P}_x$;
\bigskip

\noindent{\bf set}  $k=1$;
\bigskip

\noindent\textbf{while}  $x_k\notin S$

\hspace{1cm}
\begin{minipage}{\dimexpr\textwidth-2cm}

\medskip

  {\bf if} $x_k^{(1)}<-1$ {\bf then}  
  
  \medskip

  \hspace{1cm}
  \begin{minipage}{\dimexpr\textwidth-4cm}

\medskip
  {\bf sample} $x_{k+1}$ from the distribution of $X_{\tau_{-1}^{\wedge}}$ under  $\mathbb{P}_{x_k}$;
  
  \medskip
  {\bf set:} $k=k+1$;
  
  \end{minipage}
    \medskip

{\bf end if}

\bigskip

  {\bf if} $x_{k}^{(1)}>1$ {\bf then}
  \medskip

  \hspace{1cm}
  \begin{minipage}{\dimexpr\textwidth-4cm}
  
    {\bf sample} $x_{k+1}$ from the distribution of $X_{\tau_{1}^{\vee}}$ under $\mathbb{P}_{x_k}$;
    \medskip
    
   {\bf set} $k=k+1$;
   \medskip
   
 \end{minipage}
     \medskip
     
{\bf end if}

\end{minipage}
\bigskip

\noindent{\bf end while}
\bigskip

\noindent{\bf return} $N= k$;
\bigskip

\noindent{\bf return} $X_{\tau_S} = x_k$;
\bigskip

\hrule

\bigskip

\noindent More generally, the Algorithm above stores the information $(x_1, \cdots, x_N)$, which corresponds precisely to a simulation of $(X_{\sigma_1}, \cdots, X_{\sigma_N})$.
\smallskip

One of the problems with Algorithm 1 is that, in the setting $\alpha\in(0,1)$, the while loop may never end on account of the fact that $\mathbb{P}_x(\tau_S=\infty)>0$, by Corollary 1.2 of \cite{watson}.
Nonetheless, it is possible to `force' a stable process to hit $S$ in the case $\alpha \in (0,1)$ improving the efficiency of the walk-on-half-spaces algorithm without compromising the exactness of the algorithm.

We recall from Chapter 12.2 of  \cite{KJCbook} that, when $\alpha\in(0,1)$, the change of measure
\[
\frac{\dd \mathbb{P}^\circ_x}{\dd \mathbb{P}_x}\Bigg|_{(X_s,s\leq t)} = \left|\frac{X^{(1)}_t}{x^{(1)}}\right|^{\alpha -1}
\]
corresponds to conditioning the process $X^{(1)}$ to continuously absorb at the origin in an almost surely finite time. 
It follows that 
\[
\mathbb{P}_x(X_{\tau_S}\in \dd y,\tau_S<\infty)=\frac{|x^{(1)}|^{\alpha-1}}{|y^{(1)}|^{\alpha-1}}\mathbb{P}^{\circ}_x(X_{\tau_S}\in \dd y), \qquad x^{(1)}\not\in(-1,1), y^{(1)}\in(-1,1).
\]
This suggests we can numerically reconstruct the law of $X_{\tau_S}$ by undertaking Algorithm 1 albeit with sampling taken under $\mathbb{P}^\circ$ rather than $\mathbb{P}$.
Fundamentally, this means sampling from the two distributions 
\begin{equation}\label{eq:ch5:pcirc}
    \mathbb{P}_x^{\circ}(X_{\tau_1^{\vee}}\in \dd y)=C_{\alpha,d}\frac{|y^{(1)}|^{\alpha-1}}{(x^{(1)})^{\alpha-1}}|x-y|^{-d}\frac{(x^{(1)}-1)^{\alpha/2}}{(1-y^{(1)})^{\alpha/2}}{\rm d}y, \qquad y^{(1)}<1<x^{(1)},
\end{equation}
and
\begin{equation*}
    \mathbb{P}_x^{\circ}(X_{\tau_{-1}^{\wedge}}\in \dd y)=C_{\alpha,d}\frac{|y^{(1)}|^{\alpha-1}}{(-x^{(1)})^{\alpha-1}}|x-y|^{-d}\frac{(-1-x^{(1)})^{\alpha/2}}{(y^{(1)}+1)^{\alpha/2}}{\rm d}y, \qquad x^{(1)}<-1<y^{(1)}.
\end{equation*}

\section{A numerical example}
We consider the implementation of the walk-on-half-spaces algorithm for a stable process on $\mathbb{R}^2$. 
To sample from the distribution of $X_{\tau_1^{\vee}} = (X_{\tau_1^{\vee}}^{(1)}, X_{\tau_1^{\vee}}^{(2)})$ under $\mathbb{P}_x$, $x^{(1)} > 1$, we split the procedure into two steps. 

\smallskip

First, we simulate an instance of $X_{\tau_1^{\vee}}^{(1)}$. This can be done by treating this as a first passage problem of the one-dimensional symmetric stable process $X^{(1)}$ and noting from Lemma \ref{bgrfe} that 
$$
\mathbb{E}_{x^{(1)}}\left[ f\left(\frac{1-X_{\tau_1^{\vee}}^{(1)}}{x^{(1)}-X_{\tau_1^{\vee}}^{(1)}}\right)\right] = \frac{\sin(\alpha\pi/2)}{\pi}\int_0^1   f(u) u^{-\alpha/2} (1-u)^{\alpha/2-1} \, \mathrm{d}u,
$$
which said in another way means $(1 - X_{\tau_1^{\vee}}^{(1)})(X_{\tau_1^{\vee}}^{(1)} - x^{(1)})^{-1}$, when $X^{(1)}_0 = x^{(1)}$, follows a Beta$(1-\alpha/2, \alpha/2)$ distribution. Observe that this transform is a bijection, hence we can simulate from the Beta kernel and map it back to obtain $X_{\tau_1^{\vee}}^{(1)}$.

\smallskip

Second, we use equation \eqref{ReferenceToEqnInTheproofOfThm3}, which under the event $\{X_{\tau_1^{\vee}}^{(1)} = y_1^{(1)}\}$, allows us to sample $X_{\tau_1^{\vee}}^{(2)}-x^{(2:d)}$ by sampling from a $\mathtt{Cauchy}_1( 0, |X_{\tau_1^{\vee}}^{(1)} - x^{(1)}|)$ distribution. By symmetry, the same can be done for sampling from $X_{\tau_{-1}^{\wedge}}$ under $\mathbb{P}_{x^{(1)}}$, $x^{(1)} < -1$. Except for these details on how to simulate $X_{\tau_1^{\vee}}$ and $X_{\tau_{-1}^{\wedge}}$, the rest of the algorithm implementation is straightforward, see Figure \ref{fig:algorithm}.

\smallskip 

In the numerical examples, cf. Figure \ref{fig:joint}, we consider, without loss of generality, that $x^{(2)} = 0$. To illustrate the effect of choosing different parameters on the empirical distribution of $X_{\tau_S}$ we vary the point of issue $x^{(1)}$ and the scale parameter $\alpha$. In Figure \ref{fig:joint} (left) we observe the empirical distribution of 5 million first-hitting points in the strip, for 4 combinations of parameters $x^{(1)}$ and $\alpha$. We observe that an initial point closer to the boundary of the strip leads to a sample more concentrated in the horizontal axis neighbouring the starting point. Similarly, a scale index closer to $\alpha = 2$ leads to a first-hitting point more concentrated around the starting position, but also increases the probability of hitting the strip near its boundaries. These effects can also be seen in the marginal histograms of these samples, Figure \ref{fig:joint} (centre). It is worth noting that the marginal distribution of the first coordinate is known explicitly (cf. Theorem 1.1 in \cite{watson}). The code for these simulations can be found in the following github repository \url{https://github.com/sdaphnis/WoHS/}.

\begin{figure}[h!]
    \centering

\tikzset{every picture/.style={line width=0.75pt}} 

\begin{tikzpicture}[x=0.75pt,y=0.75pt,yscale=-0.8,xscale=0.8]

\draw  [fill={rgb, 255:red, 254; green, 255; blue, 245 }  ,fill opacity=1 ][dash pattern={on 5.63pt off 4.5pt}][line width=1.5]  (114,108.44) -- (568,108.44) -- (568,297.44) -- (114,297.44) -- cycle ;
\draw [line width=2.25]    (114,109) -- (568,108.44) ;
\draw [line width=2.25]    (114,297.44) -- (568,297.44) ;
\draw [color={rgb, 255:red, 128; green, 128; blue, 128 }  ,draw opacity=1 ] (596,81.99) -- (646,81.99)(601,42) -- (601,86.44) (639,76.99) -- (646,81.99) -- (639,86.99) (596,49) -- (601,42) -- (606,49)  ;

\draw [color={rgb, 255:red, 208; green, 2; blue, 27 }  ,draw opacity=1 ]   (206,77) -- (343.5,77) ;
\draw [color={rgb, 255:red, 208; green, 2; blue, 27 }  ,draw opacity=1 ]   (337,343.5) -- (483,343.5) ;
\draw [color={rgb, 255:red, 208; green, 2; blue, 27 }  ,draw opacity=1 ]   (449,72.5) -- (483,72.5) ;

\draw  [dash pattern={on 0.84pt off 2.51pt}]  (206,77) -- (337,343.5) ;
\draw [shift={(337,343.5)}, rotate = 63.82] [color={rgb, 255:red, 0; green, 0; blue, 0 }  ][fill={rgb, 255:red, 0; green, 0; blue, 0 }  ][line width=0.75]      (0, 0) circle [x radius= 3.35, y radius= 3.35]   ;
\draw [shift={(206,77)}, rotate = 63.82] [color={rgb, 255:red, 0; green, 0; blue, 0 }  ][fill={rgb, 255:red, 0; green, 0; blue, 0 }  ][line width=0.75]      (0, 0) circle [x radius= 3.35, y radius= 3.35]   ;
\draw  [dash pattern={on 0.84pt off 2.51pt}]  (337,343.5) -- (483,72.5) ;
\draw [shift={(483,72.5)}, rotate = 298.31] [color={rgb, 255:red, 0; green, 0; blue, 0 }  ][fill={rgb, 255:red, 0; green, 0; blue, 0 }  ][line width=0.75]      (0, 0) circle [x radius= 3.35, y radius= 3.35]   ;
\draw  [dash pattern={on 0.84pt off 2.51pt}]  (483,72.5) -- (449,221.5) ;
\draw [shift={(449,221.5)}, rotate = 147.85] [color={rgb, 255:red, 0; green, 0; blue, 0 }  ][line width=0.75]    (-5.59,0) -- (5.59,0)(0,5.59) -- (0,-5.59)   ;

\draw (197,38.4) node [anchor=north west][inner sep=0.75pt]    {$X_{0}$};
\draw (345,357.4) node [anchor=north west][inner sep=0.75pt]    {$X_{\sigma _{1}} =X_{\tau _{1}^{\lor }}$};
\draw (488,41.4) node [anchor=north west][inner sep=0.75pt]    {$X_{\sigma _{2}}$};
\draw (456,209.4) node [anchor=north west][inner sep=0.75pt]    {$X_{\sigma _{3}} =X_{\tau _{S}}$};
\draw (588,16.4) node [anchor=north west][inner sep=0.75pt]    {$X^{( 1)}$};
\draw (655,74.4) node [anchor=north west][inner sep=0.75pt]    {$X^{( 2:d)}$};
\draw (75,290.4) node [anchor=north west][inner sep=0.75pt]    {$-1$};
\draw (89,103.4) node [anchor=north west][inner sep=0.75pt]    {$1$};

\end{tikzpicture}
    \caption{Representation of the Walk on half-spaces algorithm. In this illustration the process starts above the strip, the first jump out of $\mathbb{H}_{1\uparrow}$ misses the strip, we start the algorithm again taking as initial position the first hit into $\mathbb{H}_{-1\downarrow}$. The process repeats and after $N=3$ crossings the process finally enters the strip. The red lines correspond to the $\mathtt{Cauchy}_1$ horizontal displacements which add up to $X^{(2:d)}_{\tau_S}$.}
    \label{fig:algorithm}
\end{figure}
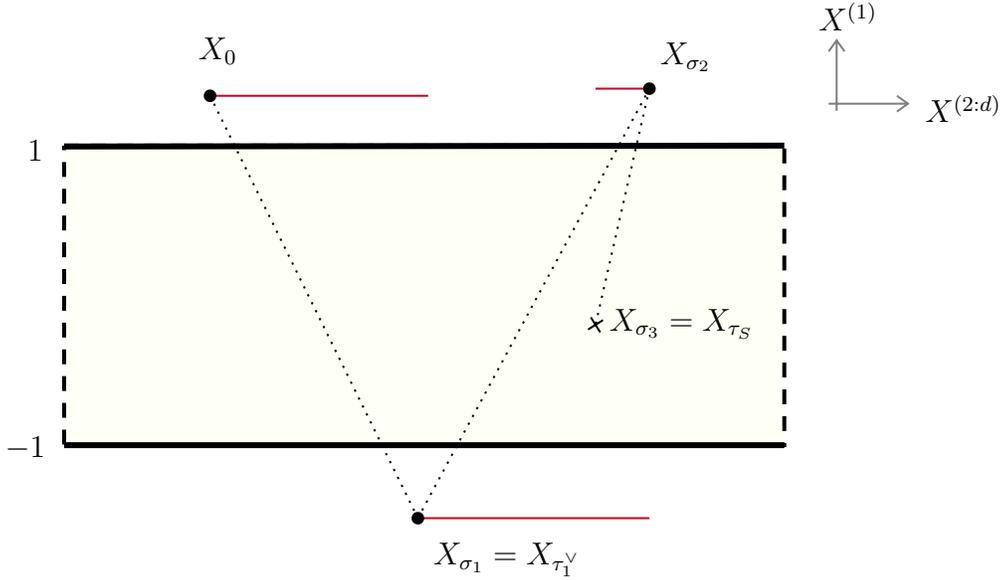

\begin{figure}[h!]
    \centering
    \includegraphics[width=0.3\textwidth]{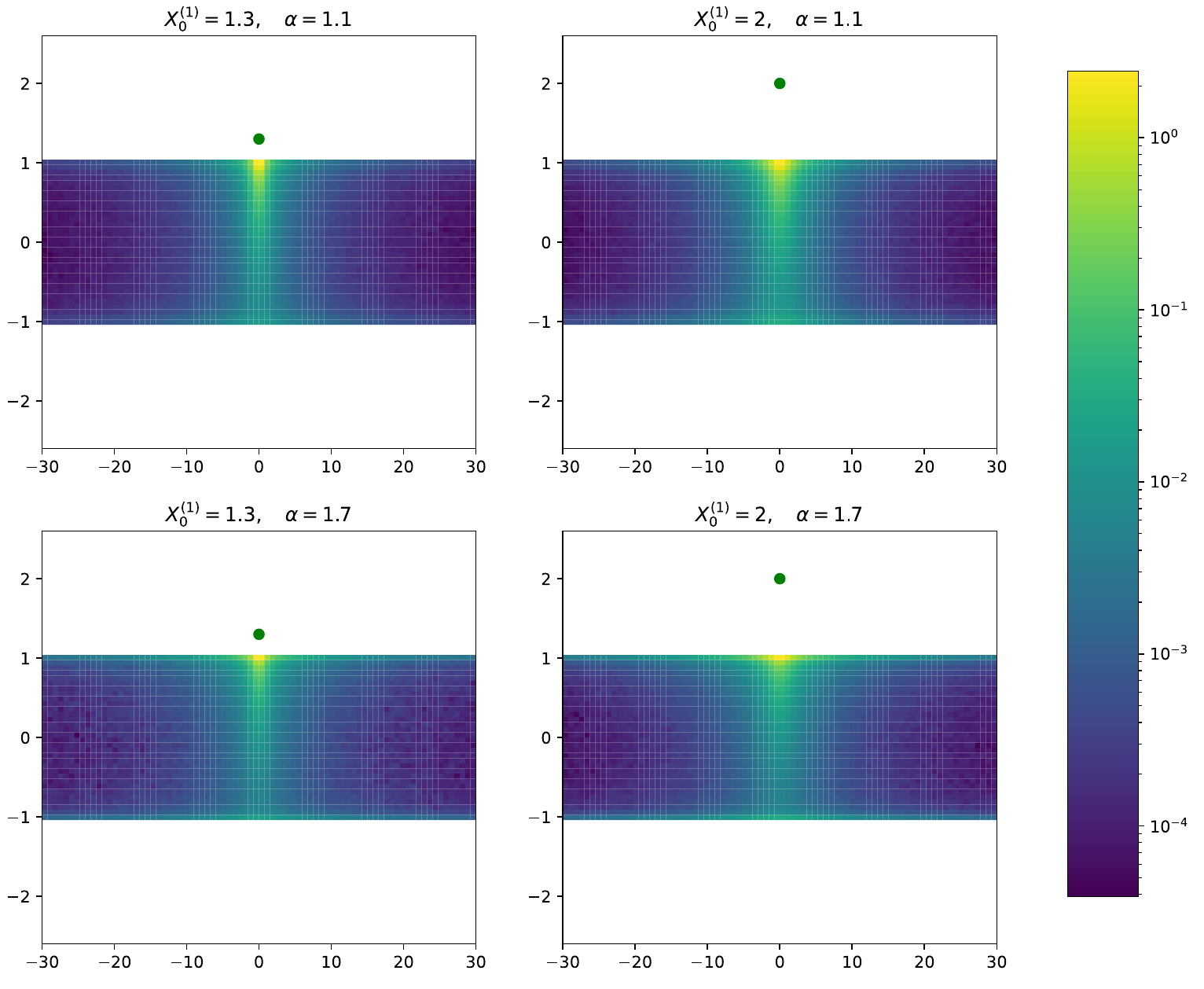}
        \includegraphics[width=0.3\textwidth, height = 0.25\textwidth]{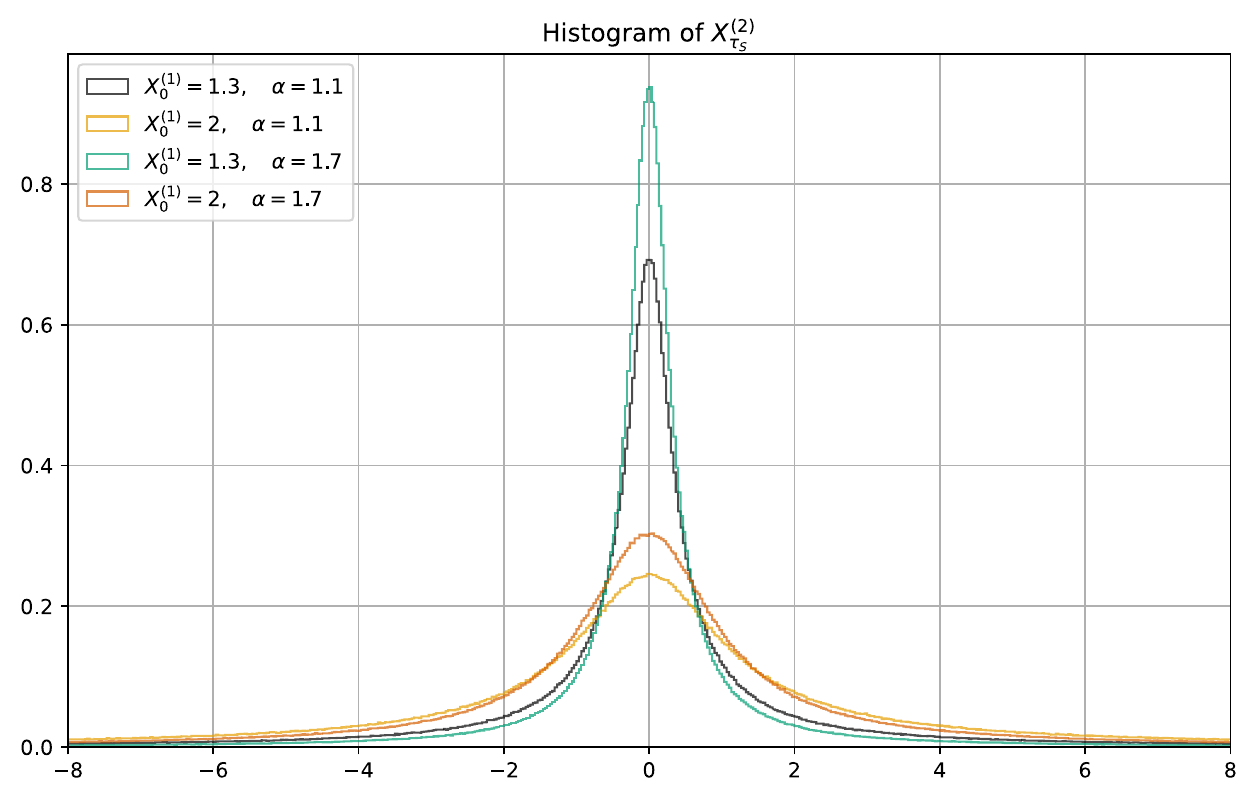}
    \includegraphics[width=0.3\textwidth]{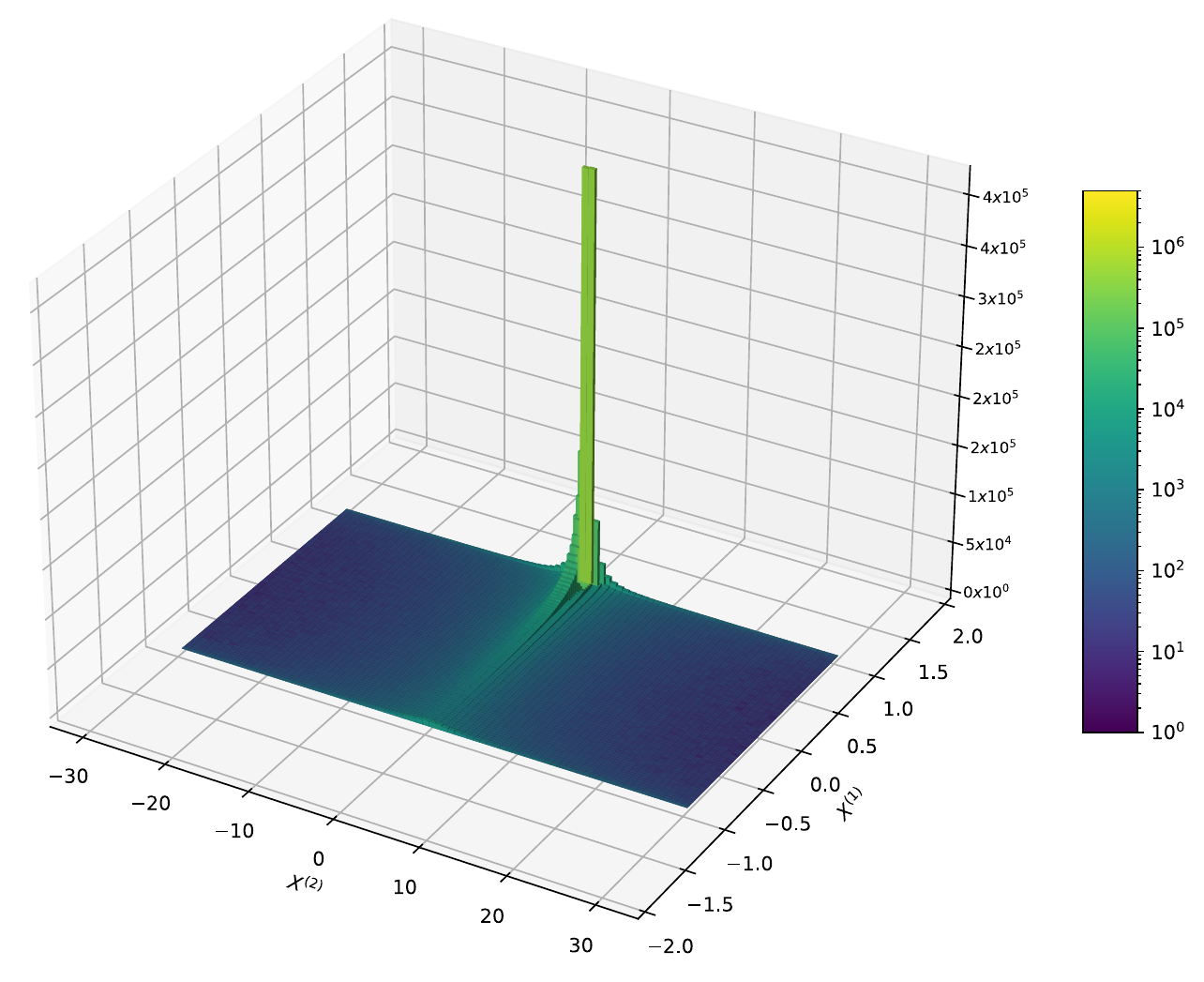}
    \caption{(left): Histograms of the simulations of first-hits into the strip for different values of $X_0=(x^{(1)},0)$ and $\alpha$. The green point represents the starting point $X_0$. 
    (centre): Histograms of the second coordinate at the instant of first-hit into the strip. Different values for the point of issue and $\alpha$ are considered. Instead of drawing solid bars, the histograms are drawn as step-like curves that outline the shape of the histogram, i.e. we see the edges of each bin without any filled area underneath. (right): A 3d histogram of the first hits into the strip.}
    \label{fig:joint}
\end{figure}

%
%
%

\section{Proofs of Lemmas \ref{cor:marginal} and \ref{bgrfe}}
In order to prove Lemmas \ref{cor:marginal} and  \ref{bgrfe} we use 
a method of converting results for first entry into a ball into first entry into a half space by expanding the radius of the sphere asymptotically so that, with re-centring, the interior of the ball becomes a half-space. We call this the {\it  flat earth theory}, in the sense that, from an infinitesimally localised perspective, the surface of a sphere appears to be that of a flat surface. 
 
\smallskip

The lemma we present below combines a geometric argument with the Lévy and scaling properties of the stable process $(X,\mathbb{P})$. 

\begin{lemma}[Flat earth approximation] \label{lem:flatearth} Let $\tau_{r}^{\oplus}=\inf\{t>0:|X_t|<r\}$  and suppose 
$f:\mathbb{R}^d\times\mathbb{R}^d\to \mathbb{R}$ is bounded and measurable. Then
    \begin{equation}
\mathbb{E}_x\left[ f(X_{\tau_r^{\vee}}, X_{\tau_r^{\vee}-})\right]=\lim_{R\to \infty}\mathbb{E}_{x+R\mathbf{e}_1}\left[f(X_{\tau_{R+r}^{\oplus}} - R\mathbf{e}_1, X_{\tau_{R+r}^{\oplus}-} - R\mathbf{e}_1),\tau_{R}^{\oplus}<\infty\right],
\end{equation}
where $\mathbf{e}_1 = (1,0,\cdots,0)\in\mathbb{R}^d$.
\end{lemma}
\begin{proof}
 We start by approximating $\mathbb{H}_{r\downarrow}=\{z\in \mathbb{R}^d:z^{(1)}<r\}$, $r>0$, by an increasing sequence of balls $\mathbb{B}_{(-R\mathbf{e}_1,R+r)}=\{z\in \mathbb{R}^d: |z+R\mathbf{e}_1|<R+r\}\subseteq \mathbb{H}_{r\downarrow}$ where  $R>0$. We have  
$$\lim_{R\to \infty}\mathbb{B}_{(-R\mathbf{e}_1,R+r)}=\mathbb{H}_{r\downarrow}.$$
To see why, we only have to show that $\mathbb{H}_{r\downarrow}\subseteq\lim_{R\to \infty}\mathbb{B}_{(-R\mathbf{e}_1,R+r)}$. Indeed, take an arbitrary point in $\mathbb{H}_{r\downarrow}$ written $y=(r-r_0,y^{(2:d)})$. An elementary calculation shows that if $R>(2r_0)^{-1}(r_0^2+|y^{(2:d)}|^2) -r$, then $y\in \mathbb{B}_{(-R\mathbf{e}_1,R+r)}$.

\begin{figure}[ht!]
    \centering
    \tikzset{every picture/.style={line width=0.75pt}} 
\begin{tikzpicture}[x=1pt,y=1pt,yscale=-1,xscale=1]
\draw  [draw opacity=0][fill={rgb, 255:red, 155; green, 155; blue, 155 }  ,fill opacity=0.09 ] (131.5,61.5) -- (422.5,61.5) -- (422.5,210.75) -- (131.5,210.75) -- cycle ;
\draw   (248.63,90.38) .. controls (248.63,74.7) and (261.33,62) .. (277,62) .. controls (292.67,62) and (305.38,74.7) .. (305.38,90.38) .. controls (305.38,106.05) and (292.67,118.75) .. (277,118.75) .. controls (261.33,118.75) and (248.63,106.05) .. (248.63,90.38) -- cycle ;
\draw   (224.88,114.13) .. controls (224.88,85.34) and (248.21,62) .. (277,62) .. controls (305.79,62) and (329.13,85.34) .. (329.13,114.13) .. controls (329.13,142.91) and (305.79,166.25) .. (277,166.25) .. controls (248.21,166.25) and (224.88,142.91) .. (224.88,114.13) -- cycle ;
\draw    (131.5,61.5) -- (422.5,61.5) ;
\draw   (204.13,134.88) .. controls (204.13,94.63) and (236.75,62) .. (277,62) .. controls (317.25,62) and (349.88,94.63) .. (349.88,134.88) .. controls (349.88,175.12) and (317.25,207.75) .. (277,207.75) .. controls (236.75,207.75) and (204.13,175.12) .. (204.13,134.88) -- cycle ;
\draw  [dash pattern={on 0.84pt off 2.51pt}]  (131.5,79) -- (422.5,79) ;
\draw  [dash pattern={on 0.84pt off 2.51pt}]  (276.5,43.5) -- (276.5,248.75) ;
\draw (278.5,51.4) node [anchor=north west][inner sep=0.75pt]  [font=\scriptsize]  {$r$};
\draw (273.5,74.4) node [anchor=north west][inner sep=0.75pt]  [font=\scriptsize]  {$0$};
\draw (382.5,182.4) node [anchor=north west][inner sep=0.75pt]  [font=\scriptsize]  {$\mathbb{H}_{r\downarrow }$};
\draw (281,173.4) node [anchor=north west][inner sep=0.75pt]  [font=\tiny]  {$\mathbb{B}_{(-Re1,R+r)}$};
    \end{tikzpicture}
    \caption{Approximation of $\mathbb{H}_{r\downarrow}$ by $\mathbb{B}_{(-R\mathbf{e}_1,R+r)}$.}
    \label{fig:flat earth}
\end{figure}
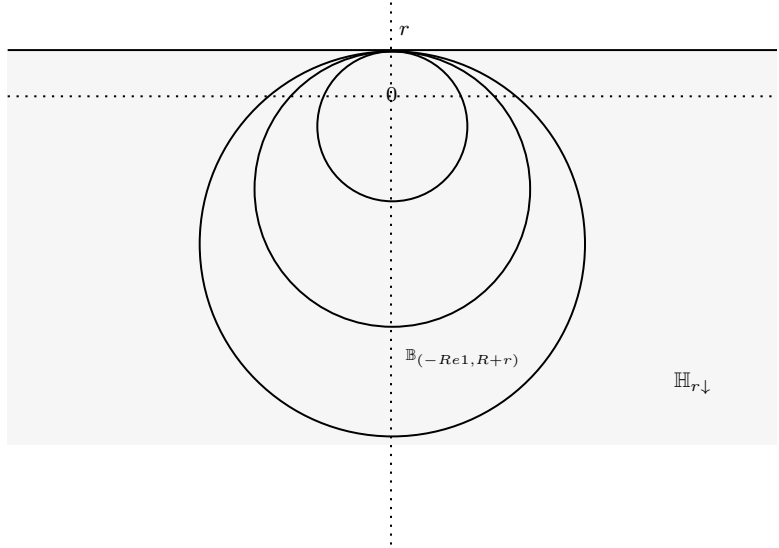

The next assertion is that we can define a decreasing sequence of stopping times which will limit to $\tau_r^{\vee}$. If $\tau_{\mathbb{B}_{R+r}}=\inf\{t>0:X_t\in  \mathbb{B}_{(-R\mathbf{e}_1,R+r)}\}$, we claim that, for all $x\in\mathbb{R}^d$ with $x^{(1)}>r$, $ \mathbb{P}_x$-almost surely it holds that
$$
\tau_{\mathbb{B}_{R+r}}\downarrow\tau_r^{\vee},
$$
as $R\to\infty$.

\smallskip

The first thing we need to point out is that $(\tau_{\mathbb{B}_{R+r}}, R>0)$ is a decreasing sequence of stopping times, hence its limit is a stopping time. For $\mathbb{P}_x$-almost every $\omega$ let $y(\omega)=X_{\tau_r^{\vee}(\omega)}(\omega)$. Then there exists a $R_0(\omega)>0$ such that $y(\omega)\in \mathbb{B}_{(-R_0(\omega)\mathbf{e}_1,R_0(\omega)+r)}$, implying by convexity that $\tau_{\mathbb{B}_{R+r}}(\omega)=\tau_r^{\vee}(\omega)$ for all $R>R_0.$ In which case $X_{\tau_r^{\vee}(\omega)}(\omega) = X_{\tau_{\mathbb{B}_{R+r}}(\omega)}(\omega)$ and $X_{\tau_r^{\vee}(\omega)-}(\omega) = X_{\tau_{\mathbb{B}_{R+r}-}(\omega)}(\omega)$.
\smallskip

Hence, for any $x\in \mathbb{H}_{r\uparrow}$ appealing to dominated convergence (as $f$ is bounded), 
\begin{align*}
\mathbb{E}_x\left[ f(X_{\tau_r^{\vee}}, X_{\tau_r^{\vee}-})\right]&=\mathbb{E}_x\left[\lim_{R\to \infty} f(X_{\tau_{\mathbb{B}_{R+r}}}, X_{\tau_{\mathbb{B}_{R+r}}-})\right]\\
&=\lim_{R\to \infty}\mathbb{E}_x\left[f(X_{\tau_{\mathbb{B}_{R+r}}}, X_{\tau_{\mathbb{B}_{R+r}}-}),\tau_{\mathbb{B}_{R+r}}<\infty\right]\\
&=\lim_{R\to \infty}\mathbb{E}_{x+R\mathbf{e}_1}\left[f(X_{\tau_{R+r}^{\oplus}} - R\mathbf{e}_1, X_{\tau_{R+r}^{\oplus}-}- R\mathbf{e}_1),\tau_{R+r}^{\oplus}<\infty\right],
\end{align*}
where  the last equality is given by stationary and independent increments.
\end{proof}

Next we give the proof of Lemma \ref{bgrfe} which is almost trivial in light of Lemma \ref{lem:flatearth}. We leave the proof of Lemma  \ref{cor:marginal} to the reader, pointing to Corollary 1.4 of \cite{DEEP3} as the starting expression from which limits can be taken.

\begin{proof}
Applying Lemma \ref{lem:flatearth} using the explicit formula for the distribution of the first hit inside a ball found in \cite{blumenthal1961distribution}, we have that

\begin{align*}
    \mathbb{E}_x[f(X_{\tau_r^{\vee}})]&=\lim_{R\to \infty}\mathbb{E}_{x+R\mathbf{e}_1}\left[f(X_{\tau_{R+r}^{\oplus}} - R\mathbf{e}_1 ); \tau_{R+r}^{\oplus}<\infty\right]\\
    & =\lim_{R\to \infty} C_{\alpha,d} \int_{|y|<R+r}  f(y - R\mathbf{e}_1)\frac{|(R+r)^2-|x+R\mathbf{e}_1|^2|^{\alpha/2}}{|(R+r)^2-|y |^2|^{\alpha/2}}|(x+R\mathbf{e}_1)-y|^{-d} {\rm d}y \\
     & =\lim_{R\to \infty} C_{\alpha,d} \int_{|z+R\mathbf{e}_1|<R+r}   f(z)\frac{|(R+r)^2-|x+R\mathbf{e}_1|^2|^{\alpha/2}}{|(R+r)^2-|y +R\mathbf{e}_1|^2|^{\alpha/2}}|x -z|^{-d} {\rm d}y \\
    &=C_{\alpha,d} \int_{\mathbb{H}_{r\downarrow}}  f(z)|x-y|^{-d}\lim_{R\to \infty}\frac{|2R(r-x^{(1)})+r^2-|x|^2|^{\alpha/2}}{|2R(r-y^{(1)})+r^2-|y|^2|^{\alpha/2}}{\rm d}z\\
    &=C_{\alpha,d}\int_{\mathbb{H}_{r\downarrow}}  f(z)|x-y|^{-d}\frac{|r-x^{(1)}|^{\alpha/2}}{|r-y^{(1)}|^{\alpha/2}}{\rm d}y,
\end{align*}
which completes the proof. 
\end{proof}

\begin{remark}\rm
To complete this section, we address the rather obvious question as to why flat Earth theory cannot be used to address Theorem \ref{thm:pcr}. In theory this may well be possible but we could not see a way to ensure that the point of closest reach to the origin preceding first entry to the ball $\mathbb{B}_{(-R\mathbf{e}_1,R+r)}$ asymptotically converges to $X_{G(\tau^\vee_r)}$. For this reason, we take a different approach to proving Theorem \ref{thm:pcr}, which uses the  so-called directional excursion theory in the spirit of \cite{burdzy}.
\end{remark}

\section{Directional excursion theory and proof of Theorem \ref{thm:pcr}}
Let $\underline{X}^{(1)}_t:=\inf_{s\leq t}X^{(1)}_s$ be the running infimum of the first coordinate of $X$, and $R^{(1)}_t:=X^{(1)}_t-\underline{X}^{(1)}_t$ the reflection at the infimum of $X^{(1)}$. Under $\mathbb{P}_x$, the process $R^{(1)}=(R^{(1)}_t:t\geq 0)$ is a strong Markov process for which $0$ is regular with respect to $(0,\infty)$. Moreover, the joint process $(R^{(1)}_t,X _t)_{t\geq 0}$ is a $ \mathbb{R}^+\times\mathbb{R}^d$-valued strong Markov process with transition semigroup $$P_tf(u, x):=\mathbb{E}_x\left[ f(u\vee R^{(1)}_t, X_t) \right],$$
for every $t>0$ and every nonnegative measurable function $f: \mathbb{R}^+\times \mathbb{R}^{d}\to \mathbb{R}^+$. We shall work with the canonical realization of $(R^{(1)}_t, X_t)_{t\geq 0}$ on the sample space $\mathbb{D}(\mathbb{R}^+\times \mathbb{R}^{d})$ of càdlàg paths mapping $\mathbb{R}^+$ to $\mathbb{R}^d$ with appended cemetery state, which is where the path is sent at its lifetime (which may be infinite).
\smallskip

We define $M:=\{t\geq 0: R^{(1)}_t=0\}$ and $\bar{M}$ its closure in $\mathbb{R}^+$. The set $\mathbb{R}\setminus \bar{M}$ is an open set and can be written as a union of intervals. We use $\mathtt G$ and $\mathtt D$ to denote the sets of left and right endpoints of such intervals. If we define $\mathcal{R}:=\inf\{t>0:t\in \bar{M}\}$, the downward regularity of $X^{(1)}$ implies that for every $x\in \mathbb{R}^{d}$, with probability one $\mathbb{P}_{x}(\mathcal{R}=0)=1$.
\smallskip

By Theorem 4.1 in \cite{maisonneuve} and stationary independent increments, there exist a continuous additive functional $t\mapsto L_t$ of $(R^{(1)}, X)$ carried by $\{0\}
\times \mathbb{R}^{d}$ and a kernel $\mathtt{N}$ from $ \mathbb{R}^+\times \mathbb{R}^{d}$ into $\mathbb{D}( \mathbb{R}^+\times \mathbb{R}^d)$ satisfying $\mathtt{N}_{x}(\mathcal{R}=0)=0$ and $\mathtt{N}_{x}\left[1-{\rm e}^{-\mathcal{R}} \right]\leq 1$ such that for any nonnegative predictable process $Z$ and any nonnegative function $f$ which is measurable with respect to $\sigma((R^{(1)}_t, X_t)_{t\geq 0})$  then 
\begin{equation}
    \mathbb{E}_{x}\left[ \sum_{s\in G}Z_sf\circ \theta_s \right]=\mathbb{E}_{x}\left[\int_0^{\infty} Z_s\,\mathtt{N}_{X_s}(f) \dd L_s \right],
    \label{generalexcursioncomp}
\end{equation}
where $\theta_s$ is the Markov shift operator and  $\mathtt{N}_{x}(f)=\int f \dd \mathtt{N}_{x}$. In the setting that $f$ depends only on the process $(R^{(1)}, X^{(2:d)})$ we can reduce the form of $\mathtt{N}_{x} = \mathtt{N}_{(x^{(1)}, x^{(2:d)})}$ to being dependent only on $x^{(2:d)}$, in which case we will write it as $\mathbb{N}_{x^{(2:d)}}$.
\smallskip

To define the excursion process we start by expressing the endpoints of the intervals of excursions. Let $\mathtt{G}=\cup_{t>0}\mathtt{g}_t$ and $\mathtt{D}=\cup_{t>0}\mathtt{d}_t$ where $\mathtt{g}_t:=\sup\{s<t:X^{(1)}_s=\underline{X}^{(1)}_{s}\}$ and $\mathtt{d}_t:=\inf\{s>t:X^{(1)}_s=\underline{X}^{(1)}_{s}\}$.
Hence, for all $t>0$ such that $\mathtt{d}_t>\mathtt{g}_t$, the excursions of $X$ from its directional minimum in the first axis, or equivalently the excursions of $(X^{(1)}-\underline{X}^{(1)},X^{(2:d)})$ from the set $\{0\}\times \mathbb{R}^{d-1}$ are encoded in the following process

\begin{equation}\label{excfrominf}
    (\epsilon^{(1)}_{\mathtt{g}_t}(s),\epsilon^{(2:d)}_{\mathtt{g}_t}(s)):=(X^{(1)}_{\mathtt{g}_t+s}-X^{(1)}_{\mathtt{g}_t},X^{(2:d)}_{\mathtt{g}_{t}+s}) \quad 0\leq s \leq \mathtt{d}_t-\mathtt{g}_t=:\zeta_{\mathtt{g}_t}.
\end{equation} 
Such excursions live in the space $\mathbb{D}(\mathbb{R}^+\times{\mathbb{R}^{d-1}})$ of càdlàg paths in $\mathbb{R}^+\times{\mathbb{R}^{d-1}}$,
such that
$$(\epsilon^{(1)},\epsilon^{(2:d)})=\left((\epsilon^{(1)}(t),\epsilon^{(2:d)}(t)):t\leq \zeta \right) \text{ with lifetime }\zeta=\inf\{s>0:\epsilon^{(1)}(s)<0\}$$ such that 
$\left(\epsilon^{(1)}(0),\epsilon^{(2:d)}(0)\right) \in \{0\}\times \mathbb{R}^{d-1},$  $(\epsilon^{(1)}(s),\epsilon^{(2:d)}(s))\in (0,\infty)\times \mathbb{R}^{d-1}$ and  $\left(\epsilon^{(1)}(\zeta),\epsilon^{(2:d)}(\zeta)\right) \in (-\infty,0)\times \mathbb{R}^{d-1}.$
\smallskip

Reinterpreting \eqref{generalexcursioncomp} get the so-called last exit formula  
    \begin{align} \notag
        \mathbb{E}_{(x^{(1)},x^{(2:d)})}&\left[ \sum_{g\in \mathtt{G}} F\left(X_s,  s<g\right)H\left((\epsilon^{(1)}_g,\epsilon^{(2:d)}_g)\right) \right]\\ 
        &=\mathbb{E}_{(x^{(1)},x^{(2:d)})}\left[ \int_{0}^{\infty} F\left(X_s,  s<t\right)\mathbb{N}_{X^{(2:d)}_t}\left(H(\epsilon^{(1)},\epsilon^{(2:d)})\right){\rm d} L_t \right],
        \label{eq:exitformula:levy}
    \end{align}
    where $F$ is non-negative and  bounded  on the space of càdlàg paths on $\mathbb{R}^d$ and $H$ is measurable on $\mathbb{D}(\mathbb{R}^+\times{\mathbb{R}^{d-1}})$.
    Under the measures $(\mathbb{N}_{x^{(2:d)}}, x^{(2:d)}\in \mathbb{R}^{d-1})$, the process $(\epsilon^{(1)}(s),\epsilon^{(2:d)}(s);s<\zeta)$ is Markovian with the same semigroup as $X=(X^{(1)},X^{(2:d)})$ killed at its first hitting time of $(-\infty,0]\times \mathbb{R}^{d-1}$.
\smallskip

The couple $(L,\mathbb{N}_\cdot)$ is called an exit system. This pair is not unique, but once $L$ is chosen, the measures, $(\mathbb{N}_{x^{(2:d)}}, x^{(2:d)}\in \mathbb{R}^{d-1})$ are determined but for a $L$-neglectable set, i.e. a set $A$ such that $$\mathbb{E}_{x,\theta}\left[ \int_0^{\infty}\mathbf{1}\left((X^{(1)}_s,X^{(2:d)}_s)\in A\right){\rm d}L s\right]=0.$$
\smallskip

\begin{proof}[Proof of Theorem \ref{thm:pcr}]
Let $x\in \mathbb{H}_{0\uparrow}$, the exit formula \eqref{eq:exitformula:levy} now gives us, for bounded measurable $g$ on $\mathbb{H}_{0\uparrow}$,
\begin{align}\nonumber
    \mathbb{E}_x\left[ g(X_{G(\tau_0^{\vee})})\right]&=\mathbb{E}_x\left[\sum_{t \in \mathtt{G}}g(X_t)\mathbf{1}(X^{(1)}_t+\epsilon^{(1)}_t(\zeta)<0)\right]\\ \nonumber
    &=\mathbb{E}_x\left[\int_0^{\infty}g(X_t)\mathbb{N}_{X_t^{(2:d)}}(\epsilon^{(1)}_t(\zeta)<-y)\big|_{y = X^{(1)}_t} {\rm d}L_t \right].
\end{align}
In the above expression, the event in the excursion measure depends only on the first coordinate of the excursion process. Hence it is independent of the excursion measure index $X^{(2:d)}_t$ and can be written in terms of the jump measure, say $\Pi_{-}^{(1)}$,  of the descending ladder  process associated to $X^{(1)}$, which is known to be an $\alpha/2$-stable subordinator. More precisely, 
\[
\mathbb{N}_{x^{(2:d)}}(\epsilon^{(1)}_t(\zeta)<-y) = \Pi^{(1)}_-(y,\infty) = \frac{1}{\Gamma(1-(\alpha/2))}y^{-\alpha/2}, \qquad y>0.
 \]

As such, we have
\begin{align}\nonumber
  \mathbb{E}_x\left[ g(X_{G(\tau_0^{\vee})})\right]
  &=\mathbb{E}_x\left[\int_0^{\infty}g(X_t)\frac{1}{\Gamma(1-(\alpha/2))}(X^{(1)}_t)^{-\alpha/2} {\rm d}L_t \right]\\ \nonumber
    &=\frac{1}{\Gamma(1-(\alpha/2))} \mathbb{E}_x\left[\int_0^{\infty}g(H_t^{(1)-},H^{(2:d)-}_{t})(H^{(1)-}_t)^{-\alpha/2} {\rm d}t \right]\\ 
    &=\frac{1}{\Gamma(1-(\alpha/2))}\int_{\{y^{(1)}<x^{(1)}\}}g(y^{(1)},y^{(2:d)})(y^{(1)})^{-\alpha/2}U_x^-({\rm d}y),
    \label{potential}
\end{align} 
where $H^{(1)-}_t: =X^{(1)}_{L^{-1}_t}$ and $H^{(2:d)-}_{t}: =  X^{(2:d)}_{L_t^{-1}}$, $t\geq0$, and we define the descending ladder potential measure
\begin{equation}
    U_x^-({\rm d}y):=\int_0^{\infty}  \mathbb{P}_x\left((H_t^{(1)-},H^{(2:d)-}_{t})\in {\rm d}y\right){\rm d}t \quad y^{(1)}<x^{(1)}, y^{(2:d)}\in \mathbb{R}^{d-1}.
    \label{potential-d}
\end{equation}
We note that the integral on the right-hand side of \eqref{potential} is over the whole of $\mathbb{R}^d$ with the restriction that $\{y^{(1)}<x^{(1)}\}$. In the sequel, we will see many such integrals where it is implicit that we are integrating over $\mathbb{R}^d$ and only the restriction over this space is indicated in the delimiter.
\smallskip

In conclusion, we showed that the law of $X_{G(\tau_0^{\vee})}$ under $\mathbb{P}_x$ (up to a constant) is nothing but  $U_x^-({\rm d}y)(y^{(1)})^{-\alpha/2}$. 
We therefore focus on identifying $U_x^-({\rm d}y)$. The proof we follow is guided by the calculations leading to Theorem 1.1 in \cite{DEEP3}.
\smallskip

Let us define, for $x\in \mathbb{R}^d$ such that $x^{(1)}>r>0$, $\delta>0$, and continuous, positive and bounded $f$ on $\mathbb{R}^+\times \mathbb{R}^{d-1}$,
\begin{equation}
    \Delta_r^{\delta}f(x)=\frac{1}{\delta}\mathbb{E}_{x}\big[f(r,X^{(2:d)}_{G(\tau_0^{\vee})}), X^{(1)}_{G(\tau_0^{\vee})} \in [r-\delta, r]\big].
    \label{deltadef}
\end{equation}
We want to  establish  the limit as $\delta \to 0$. In order to do so, we start by conditioning on the first entry into the half-space $\mathbb{H}_{r\downarrow}$, this can be done by appealing to Lemma \ref{bgrfe}, which we have already proved. 
We have
\begin{align} \nonumber
    \Delta_r^{\delta}f(x)&=\frac{1}{\delta} \int_{\{r-\delta \leq y^{(1)} \leq r\}} \mathbb{P}_x(X_{\tau_r^{\vee}}\in {\rm d}y) \mathbb{E}_{y}\big[f(r,X^{(2:d)}_{G(\tau_0^{\vee})}), X^{(1)}_{G(\tau_0^{\vee})} \in [r-\delta, y^{(1)}]\big]\\ \nonumber
    &=\frac{1}{\delta}C_{\alpha,d} \int_{\{ r-\delta \leq y^{(1)} \leq r\}}{\rm d}y |x-y|^{-d}\frac{|r-x^{(1)}|^{\alpha/2}}{|r-y^{(1)}|^{\alpha/2}}\mathbb{E}_{y}\big[f(r,X^{(2:d)}_{G(\tau_0^{\vee})}), X^{(1)}_{G(\tau_0^{\vee})} \in [r-\delta, y^{(1)}]\big]\\ \nonumber
    &=\frac{1}{\delta}C_{\alpha,d} \int_{\{r-\delta \leq y^{(1)} \leq r\}}{\rm d}y |x-y|^{-d}\frac{|r-x^{(1)}|^{\alpha/2}}{|r-y^{(1)}|^{\alpha/2}}\\
    &\hspace{4cm}\frac{1}{\Gamma(1-(\alpha/2))} \int_{\{r-\delta<z^{(1)}\leq y^{(1)}\}}f(r,z^{(2:d)})(z^{(1)})^{-\alpha/2}U_y^-({\rm d}z),
    \label{innerintegral}
\end{align}
where the third equality is given by \eqref{potential}. 

\smallskip

Next we provide a limiting approximation for  the inner integral  of \eqref{innerintegral} as $\delta \to 0$. From now on we will write $\tilde{U}_x^-({\rm d}z)=(z^{(1)})^{-\alpha/2}U_x^-({\rm d}z)$, $z\in\mathbb{R}^{d}$ and recall that this is proportional to the law of $X_{G(\tau_0^{\vee})}$ under $\mathbb{P}_x$.
We claim that  
\begin{equation}
    \lim_{\delta \to 0}  \sup_{\{y \in {\mathbb{R}^d}: y^{(1)} \in (r-\delta,r]\}}\bigg \lvert \frac{ \int_{\{ r-\delta<z^{(1)}\leq y^{(1)}\}}f(r,z^{(2:d)})\tilde{U}_y^-({\rm d}z)}{ \int_{\{r-\delta<z^{(1)}\leq y^{(1)}\}}\tilde{U}_y^-({\rm d}z)}-f(r,y^{(2:d)}) \bigg \rvert = 0.
    \label{first claim}
\end{equation}

Suppose that $H_{r,\delta, \varepsilon}(y)\subset\mathbb{R}^d$ is the solid cylinder defined by  $\{ x\in \mathbb{R}^{d}:  |x^{(2:d)}-y^{(2:d)}|\leq \varepsilon \text{ and } r-\delta \leq x^{(1)}\leq r\}$, denote by $H_{\varepsilon}(y) \subseteq \mathbb{R}^d$ the section of the cylinder given by $\{x^{(2:d)}\in \mathbb{R}^{d-1}:  |x^{(2:d)}-y^{(2:d)}|\leq \varepsilon \}$. 
\smallskip

Thanks to the continuity of $f$, we can  choose $\varepsilon$ such that $$\sup_{z \in H_{ \varepsilon}(y)}|f(r,y^{(2:d)})-f(r,z^{(2:d)})|< \varepsilon',$$ for some $\varepsilon'\ll 1$. Then we have
\begin{align}
&\sup_{\{y \in {\mathbb{R}^d}: y^{(1)} \in (r-\delta,r]\}}\bigg \lvert \frac{ \int_{\{r-\delta<z^{(1)}\leq y^{(1)}\}}f(r,z^{(2:d)})\tilde{U}_y^-({\rm d}z)}{ \int_{\{r-\delta<z^{(1)}\leq y^{(1)}\}}\tilde{U}_y^-({\rm d}z)}-f(r,y^{(2:d)}) \bigg \rvert\notag\\
&\leq \varepsilon' + 2\|f\|_{\infty}\sup_{\{y \in {\mathbb{R}^d}: y^{(1)} \in (r-\delta,r]\}}\frac{ \int_{\{r-\delta<z^{(1)}\leq y^{(1)}\}}\tilde{U}_y^-({\rm d}z)\mathbf{1}_{\{z^{(2:d)}\notin H_{\varepsilon}(y)\}}}{ \int_{\{r-\delta<z^{(1)}\leq y^{(1)}\}}\tilde{U}^-_{y}({\rm d}z)}.
\label{controlfraction}
\end{align}
Analysing the fraction on the right-hand side of \eqref{controlfraction}, we notice that
\begin{align}
    &\frac{ \int_{\{r-\delta<z^{(1)}\leq y^{(1)}\}}\tilde{U}_y^-({\rm d}z)\mathbf{1}_{\{z^{(2:d)}\notin H_{\varepsilon}(y)\}}}{ \int_{\{r-\delta<z^{(1)}\leq y^{(1)}\}}\tilde{U}^-_{y}({\rm d}z)}\notag\\
    &=   \frac{ \int_{\{r-\delta<z^{(1)}\leq y^{(1)}\}}U_y^-({\rm d}z)(z^{(1)})^{-\alpha/2}\mathbf{1}_{\{z^{(2:d)}\notin H_{\varepsilon}(y)\}}}{ \int_{\{r-\delta<z^{(1)}\leq y^{(1)}\}}U^-_{y}({\rm d}z)(z^{(1)})^{-\alpha/2}}\notag\\
    &=c_\alpha \frac{ \int_{\{r-\delta<z^{(1)}\leq y^{(1)}\}}U_y^-({\rm d}z)\Pi_-^{(1)}(z^{(1)} - (r-\delta),\infty) \left(1- \frac{(r-\delta)}{z^{(1)}})\right)^{\alpha/2}\mathbf{1}_{\{z^{(2:d)}\notin H_{\varepsilon}(y)\}}}{ \int_{\{r-\delta<z^{(1)}\leq y^{(1)}\}}U^-_{y}({\rm d}z)(z^{(1)})^{-\alpha/2}}\notag\\
     &\leq c_\alpha \frac{ \int_{\{r-\delta<z^{(1)}\leq y^{(1)}\}}U_y^-({\rm d}z)\Pi_-^{(1)}(z^{(1)} - (r-\delta),\infty) \left(1- \frac{(r-\delta)}{y^{(1)}}\right)^{\alpha/2}\mathbf{1}_{\{z^{(2:d)}\notin H_{\varepsilon}(y)\}}}{ \int_{\{r-\delta<z^{(1)}\leq y^{(1)}\}}U^-_{y}({\rm d}z)(y^{(1)})^{-\alpha/2}}\notag\\
    &=c_\alpha\frac{\mathbb{P}_y(H^{(2:d)-}_{\sigma_{r-\delta}^{\vee}-}\notin H_{\varepsilon}(y) )}{\left(1- \frac{(r-\delta)}{y^{(1)}}\right)^{-\alpha/2}(y^{(1)})^{-\alpha/2}U^{(1)-}_{y^{(1)}}(r-\delta, y^{(1)}]},
    \label{mess}
\end{align}
where $c_\alpha = \Gamma(1-(\alpha/2))$,
$\sigma_{r-\delta}^{\vee}=\inf\{t>0:H^{(1)-}_t<r-\delta\}$ and, recalling $y = (y^{(1)}, y^{(2:d)})$,
\[
U_{y^{(1)}}^{(1)-}({\rm d}z^{(1)}):=\int_0^{\infty}  \mathbb{P}_{y}(H_t^{(1)-}\in {\rm d}z^{(1)}){\rm d}t \quad z^{(1)}<y^{(1)}.
\]

As $y\in [r-\delta,r]\times H_{\varepsilon}(y)$, if we denote by $\nu_{\varepsilon}=\inf\{t>0: H_t^{(2:d)-}\notin H_{\varepsilon}(y)\}$, which only depends on $y$ through $y^{(2:d)}$, by the right-continuity of the paths of $X$ we conclude that, for the numerator in \eqref{mess},
$$
\lim_{\delta \to 0}  \sup_{\{y \in {\mathbb{R}^d}: y^{(1)} \in (r-\delta,r]\}}\mathbb{P}_y(H^{(2:d)-}_{\sigma_{r-\delta}^{\vee}-}\notin H_{\varepsilon}(y))
\leq \lim_{\delta \to 0}  \sup_{\{y \in {\mathbb{R}^d}: y^{(1)} \in (r-\delta,r]\}}\mathbb{P}_y(\nu_{\varepsilon}<\sigma_{r-\delta}^{\vee})=0.
$$

On the other hand, recalling (cf. Chapter 3 of \cite{KJCbook}) that  the known descending ladder height occupation measure of a one-dimensional symmetric stable process (cf. Chapter 3 of \cite{KJCbook}) gives us
\begin{equation}
U_{y^{(1)}}^{(1)-}({\rm d}v) = \frac{1}{\Gamma(\alpha/2)}(y^{(1)}-v)^{(\alpha/2)-1},\qquad  v\leq y^{(1)},
\label{stablesubpot}
\end{equation} 
for the denominator of \eqref{mess} we have 
\begin{align*}
&\left(1- \frac{(r-\delta)}{y^{(1)}}\right)^{-\alpha/2}(y^{(1)})^{-\alpha/2}U^{(1)-}_{y^{(1)}}(r-\delta, y^{(1)}]\notag\\
&=\left(1- \frac{(r-\delta)}{y^{(1)}}\right)^{-\alpha/2}(y^{(1)})^{-\alpha/2} \int_0^{y^{(1)}-(r-\delta)} \frac{u^{(\alpha/2)-1}}{\Gamma(\alpha/2)}  {\rm d}u\notag\\
&=
\frac{1}{\Gamma(1+(\alpha/2))}.
\end{align*}
Hence the denominator of \eqref{mess} is nothing but a constant. 

\smallskip

The rigth-hand side of \eqref{mess} therefore goes to zero uniformly for $y^{(1)}\in(r-\delta, r]$ as $\delta\to0$, which, from \eqref{controlfraction}, means that we have proved the claim \eqref{first claim}.
%
%
%

\smallskip

Now returning to \eqref{innerintegral}, using  \eqref{first claim}, for small $\delta$ we have 
\begin{align}\nonumber
\Delta_r^{\delta}f(x)&=\frac{1}{\delta}C_{\alpha,d} \int_{\{ r-\delta \leq y^{(1)} \leq r\}}{\rm d}y |x-y|^{-d}\frac{|r-x^{(1)}|^{\alpha/2}}{|r-y^{(1)}|^{\alpha/2}}\\ \nonumber 
& \hspace{2cm}\frac{1}{\Gamma(1-(\alpha/2))}\int_{\{r-\delta<z^{(1)}\leq y^{(1)}\}} \bigg(f(r,z^{(2:d)})-f(r,y^{(2:d)})+f(r,y^{(2:d)}) \bigg)\tilde{U}_y^-({\rm d}z)\\ \nonumber
        &= D(\varepsilon)\Delta_r^{\delta}1(x)\\ 
    & \hspace{10pt} + \frac{1}{\delta}C_{\alpha,d} \int_{\{ r-\delta \leq y^{(1)} \leq r\}}{\rm d}y |x-y|^{-d}\frac{|r-x^{(1)}|^{\alpha/2}}{|r-y^{(1)}|^{\alpha/2}}  f(r,y^{(2:d)})\mathbb{P}_y(X_{G(\tau_0^{\vee})}^{(1)}>r-\delta),
    \label{limdelta}
\end{align}

where $|D(\varepsilon)|<\varepsilon$. 
\smallskip

We deal with the two terms on the right-hand side of \eqref{limdelta} sequentially. First, we note from the original definition in \eqref{deltadef} that 
\begin{align}
\lim_{\delta\to0}\Delta_r^{\delta}1(x) &= \frac{1}{\delta}\mathbb{P}_{x}(X^{(1)}_{G(\tau_0^{\vee})} \in [r-\delta, r]) \notag\\
&=\lim_{\delta\to0}\frac{1}{\delta} \int_{r-\delta}^r \frac{(x^{(1)}-u)^{(\alpha/2)-1}u^{-\alpha/2}}{\Gamma(\alpha/2)\Gamma(1-(\alpha/2))}\dd u\notag\\
&=\frac{(x^{(1)}-r)^{(\alpha/2)-1}r^{-\alpha/2}}{\Gamma(\alpha/2)\Gamma(1-(\alpha/2))}.
\label{Delta1limit}
\end{align}
where, we have appealed to the one-dimensional analogue of \eqref{potential}; using $U^{(1)-}_{y_0}$ in place of $U^+_y$. The limit in \eqref{Delta1limit} tells us that the first term on the right-hand side of \eqref{limdelta} can be ignored as $\delta\to0$. 

\smallskip

 For the second term on the right-hand side of \eqref{limdelta}, 
we note that its limit as $\delta\to0$ is equal to 
\begin{align}
&\lim_{\delta\to0} \Delta_r^{\delta}f(x) = C_{\alpha,d}\int_{\mathbb{R}^{d-1}}\dd y^{(2:d)}\frac{|r-x^{(1)}|^{\alpha/2}f(r,y^{(2:d)})}{|(x^{(1)}-r)^2+(x^{(2:d)}-y^{(2:d)})|^{d/2}} \notag\\
&\hspace{7cm}\times\lim_{\delta\to0}\frac{1}{\delta} \int_{r-\delta}^r{\rm d}y^{(1)} \frac{ \mathbb{P}_y(X_{G(\tau_0^{\vee})}^{(1)}>r-\delta)}{(r-y^{(1)})^{\alpha/2}}. 
\label{lastlimit}
\end{align}
For the limit on the right-hand  of \eqref{lastlimit}, we can compute 
\begin{align}
&\lim_{\delta\to0}\frac{1}{\delta} \int_{r-\delta}^r{\rm d}y^{(1)} \frac{ \mathbb{P}_y(X_{G(\tau_0^{\vee})}^{(1)}>r-\delta)}{(r-y^{(1)})^{\alpha/2}}\notag\\
&=\lim_{\delta\to0}\frac{1}{\delta} \int_{r-\delta}^r{\rm d}y^{(1)} (r-y^{(1)})^{-\alpha/2}
\int_{r-\delta}^{y^{(1)}} \frac{(y^{(1)}-u)^{(\alpha/2)-1}u^{-\alpha/2}}{\Gamma(\alpha/2)\Gamma(1-(\alpha/2))}\dd u\notag\\
&=\frac{r^{-\alpha/2}}{(\alpha/2)\Gamma(\alpha/2)\Gamma(1-(\alpha/2))}
\lim_{\delta\to0}\frac{1}{\delta} \int_{r-\delta}^r{\rm d}y^{(1)} (r-y^{(1)})^{-\alpha/2}(y^{(1)}-r+\delta)^{\alpha/2}\notag\\
&=\frac{r^{-\alpha/2}}{(\alpha/2)\Gamma(\alpha/2)\Gamma(1-(\alpha/2))}
\lim_{\delta\to0}\frac{1}{\delta} \int_0^\delta v^{-\alpha/2} (\delta-v)^{\alpha/2}\dd v \notag\\
&=\frac{r^{-\alpha/2}}{ \Gamma(1+(\alpha/2))\Gamma(1-(\alpha/2))}
 \int_0^1  w^{-\alpha/2} (1- w)^{\alpha/2}\dd w\notag\\
 &=r^{-\alpha/2}.
 \label{lastintegral}
\end{align}
where the first equality uses a calculation similar to \eqref{Delta1limit}.

Putting \eqref{lastintegral} back into \eqref{lastlimit}, we have from \eqref{deltadef} that 
\begin{align}
&\mathbb{E}_x\Big[f(X^{(1)}_{G(\tau_0^{\vee})},X^{(2:d)}_{G(\tau_0^{\vee})})\Big]\notag\\
&=\int_0^{x^{(1)}}\lim_{\delta\to0}\Delta_r^{\delta}f(x)\dd r   \notag\\
&=C_{\alpha,d}\int_0^{x^{(1)}}\dd r\int_{\mathbb{R}^{d-1}}\dd y^{(2:d)}\frac{|r-x^{(1)}|^{\alpha/2}f(r,y^{(2:d)})}{r^{\alpha/2}|(x^{(1)}-r)^2+(x^{(2:d)}-y^{(2:d)})|^{d/2}} 
\end{align}
which completes the proof of Theorem \ref{thm:pcr}.
\end{proof}

\section{Proof of Theorem \ref{thm:triplelaw}}\label{prf1sect}
There are two fromulae we will need to first derive in order to prove Theorem \ref{thm:triplelaw}. For the first, 
just as we have constructed the directional excursions from the minimum in the first axis, we can construct directional excursions from the maximum in the first axis resulting in an ascending ladder process, say $(H^{(1)+}, H^{(2:d)+})$. 
If we define the renewal function of the ascending ladder process of the latter, say  $U^+_x(\dd y)$,  in the obvious way (cf. \eqref{potential-d}), then it is immediate from Theorem \ref{thm:pcr}, \eqref{potential} and symmetry that 
    
\begin{equation}\label{ascpot}
    U_x^{+}({\rm d}z)=\frac{\Gamma(d/2)}{\pi^{d/2}\Gamma(\alpha/2)}\frac{(z^{(1)}-x^{(1)})^{\alpha/2}}{|x-z|^d}{\rm d}z \quad \text{ for } x^{(1)}<z^{(1)}, z^{(2:d)}\in\mathbb{R}^{d-1}.
\end{equation}

The second formula we need is provided by the  following lemma, which  gives us an  occupation functional under the excursion measure.

\begin{lemma}\label{lem:silverstein}
For $x\in \mathbb{R}^d$ and continuous $g:\mathbb{R}^d\to \mathbb{R}$ whose support is compactly embedded in $\mathbb{H}_{r\uparrow}$ 

\begin{align}
\mathbb{N}_{x^{(2:d)}}\left(\int_0^{\zeta}g(x^{(1)}+\epsilon^{(1)}(u),\epsilon^{(2:d)}(u)){\rm d}u \right) & =\int_{\{ x^{(1)}<z^{(1)}\}}g(z)\frac{\Gamma(d/2)}{\pi^{d/2}\Gamma(\alpha/2)}\frac{(z^{(1)}-x^{(1)})^{\alpha/2}}{|x-z|^d} {\rm d}z,\notag\\
& = \int_{\{ x^{(1)}<z^{(1)}\}} g(z)  U_x^+({\rm d}z).
\label{excursionduality}
\end{align}
\end{lemma}

\begin{proof}
 By the last exit formula \eqref{eq:exitformula:levy} we note that 
\begin{align} \nonumber
&  \mathbb{E}_x\left[\int_0^{\tau_r^{\vee}}g(X_s){\rm d}s \right] \notag\\
&=\mathbb{E}_{x}\left[ \sum_{g\in \mathtt{G}}  \mathbf{1}_{\{\underline{X}_g^{(1)}>r\}}\int_g^d g( \underline{X}_g^{(1)} + \epsilon^{(1)}_g(s),\epsilon^{(2:d)}_g(s) ) \dd s  \right]\notag\\
&= \mathbb{E}_x\left[\int_0^{\infty} \mathbf{1}_{\{\underline{X}_s^{(1)}>r\}}
\mathbb{N}_{X^{(2:d)}_s}\left(\int_0^\zeta g (y + \epsilon^{(1)}(u), \epsilon^{(2:d)}(u) )\dd u\right)\Big|_{y = X^{(1)}_s}
{\rm d} L_s \right]\notag\\
&= \mathbb{E}_x\left[\int_0^{\infty} \mathbf{1}_{\{H^{(1)-}_s>r\}}
\mathbb{N}_{H^{(2:d)-}_s}\left(\int_0^\zeta g (y + \epsilon^{(1)}(u), \epsilon^{(2:d)}(u) )\dd u\right)\Big|_{y = H^{(1)-}_s}
{\rm d}s \right]\notag\\
    &=\int_{\{r<z^{(1)}<x^{(1)}\}}U_x^-({\rm d}z)\mathbb{N}_{z^{(2:d)}}\left(\int_0^{\zeta}g(z^{(1)}+\epsilon^{(1)}(u),\epsilon^{(2:d)}(u)){\rm d}u \right) \nonumber\\  
         &=\int_{\{r<z^{(1)}<x^{(1)}\}}\tilde{U}_x^-({\rm d}z)(z^{(1)})^{\alpha/2}\mathbb{N}_{z^{(2:d)}}\left(\int_0^{\zeta}g(z^{(1)}+\epsilon^{(1)}(u),\epsilon^{(2:d)}(u)){\rm d}u \right).
         \label{ocupfunc}
\end{align}

Next, we want to make use of an argument very similar to \eqref{first claim} replacing the role of $f$ there by the function 
\[
\kappa(z): =(z^{(1)})^{\alpha/2}  \mathbb{N}_{z^{(2:d)}}\left(\int_0^{\zeta}g(z^{(1)}+\epsilon(u),\epsilon^{(2:d)}(u))\right).
\]
Note that we cannot use \eqref{first claim} directly as the function $f$ there cannot take the form of $\kappa$. Nonetheless, the result in \eqref{first claim} still holds when we make the aforesaid exchange of functions, providing we can demonstrate that $\kappa(z)$ is bounded and continuous in $z$.

\smallskip

Boundedness of $\kappa$ is ensured by the compact support property of $g$. Continuity of $\kappa$ is ensured by dominated convergence, using the compact embedding of the support  of $g$, the boundedness and continuity of $g$ and the fact that 
\[
\mathbb{N}_{z^{(2:d)}}\left(\int_0^{\zeta}g(z^{(1)}+\epsilon(u),\epsilon^{(2:d)}(u))\right)
=\mathbb{N}_{0}\left(\int_0^{\zeta}g(z^{(1)}+\epsilon(u),z^{(2:d)}+\epsilon^{(2:d)}(u))\right).
\]

Hence using $\kappa$ in place of $f$ in \eqref{first claim} we have on the one hand that 
\begin{align}
(x^{(1)})^{\alpha/2}\mathbb{N}_{x^{(2:d)}}&\left(\int_0^{\zeta}g(x^{(1)}+\epsilon(u),\epsilon^{(2:d)}(u)){\rm d}u \right)=\lim_{r \uparrow x^{(1)}}\frac{\mathbb{E}_x\left[ \int_0^{\tau_r^{\vee}}g(X_s){\rm d}s \right]}{\int_{\{r<z^{(1)}<x^{(1)}\}}\tilde{U}_x^-({\rm d}z)}.
\label{onehand}
\end{align}

Although we have not calculated an expression for the numerator on the right-hand side of \eqref{onehand}, it is easily derived from Lemma \ref{cor:marginal}.  Indeed, it is easy to derive from the underlying L\'evy system of $X$ that,  for $x,y \in \mathbb{H}_{r\uparrow}$ and $z \in \mathbb{H}_{r\downarrow}$,
 \begin{align*}
 \Xi_x(\bullet,y,z)\dd y\dd z =  \mathbb{E}_x\left[ \int_0^{\tau_r^{\vee}}\mathbf{1}_{\{X_s\in \dd y\}}{\rm d}s \right]\Pi(\dd z-y)
 \end{align*}
 As such, 
 we have,  for $x,y \in \mathbb{H}_{r\uparrow}$ and $z \in \mathbb{H}_{r\downarrow}$,
 \begin{align}
 h_r^{\vee}(x,y)&:=\mathbb{E}_x\left[ \int_0^{\tau_r^{\vee}}\mathbf{1}_{\{X_s\in \dd y\}}{\rm d}s \right] \notag\\
 &= E_{\alpha,d}
  |x-y|^{\alpha-d}\int_0^{\zeta_r^{\vee}(x,y)}(u+1)^{-d/2}u^{\alpha/2-1}{\rm d}u,  
 \end{align}
 where $E_{\alpha,d} = \Gamma(d/2)/2^{\alpha}  \pi^{d/2}\Gamma(\alpha/2)^2  $ and we recall 
 \[
 \zeta_r^{\vee}(x,y):=\frac{4(x^{(1)}-r)(y^{(1)}-r)}{|x-y|^2}.
 \]

Therefore,  on the other hand we have that 
\begin{align}
&\lim_{r \uparrow x^{(1)}}\frac{\mathbb{E}_x\left[ \int_0^{\tau_r^{\vee}}g(X_s){\rm d}s \right]}{\int_{\{r<z^{(1)}<x^{(1)}\}}\tilde{U}_x^-({\rm d}z)}\notag\\
&=\lim_{r \uparrow x^{(1)}}\frac{\int_{\{x^{(1)}<z^{(1)}\}} \mathbf{1}_{\{r<z^{(1)}\}}g(z)h_r^{\vee}(x,z) {\rm d}z}{\int_{\{r<z^{(1)}<x^{(1)}\}} (z^{(1)})^{-\alpha/2}U_x^{(1)-}({\rm d}z^{(1)})}\notag\\
&=\lim_{r \uparrow x^{(1)}}\frac{\int_{\{x^{(1)}<z^{(1)}\}} \mathbf{1}_{\{r<z^{(1)}\}}g(z)E_{\alpha,d}|x-z|^{\alpha-d}\int_{0}^{\zeta_r^{\vee}(x,z)}(u+1)^{-d/2}u^{\alpha/2-1}{\rm d}u{\rm d}z}{\Gamma(\alpha/2)^{-1}\int_r^{x^{(1)}}u^{-\alpha/2}(x^{(1)}-u)^{\alpha/2-1}{\rm d}u}\notag\\
&=E_{\alpha,d}\Gamma(\alpha/2)\int_{\{x^{(1)}<z^{(1)}\}}g(z)|x-z|^{\alpha-d}\lim_{r\uparrow x^{(1)}}\frac{\int_{0}^{\zeta_r^{\vee}(x,z)}(u+1)^{-d/2}u^{\alpha/2-1}{\rm d}u}{\int_r^{x^{(1)}}u^{-\alpha/2}(x^{(1)}-u)^{\alpha/2-1}{\rm d}u}{\rm d}z\notag\\
&=E_{\alpha,d}\Gamma(\alpha/2)(x^{(1)})^{\alpha/2}\int_{\{x^{(1)}<z^{(1)}\}}g(z)|x-z|^{\alpha-d}\left(\frac{4(z^{(1)}-x^{(1)})}{|x-z|^2} \right)^{\alpha/2}  {\rm d}z\notag\\
&=2^\alpha E_{\alpha,d}\Gamma(\alpha/2)(x^{(1)})^{\alpha/2}\int_{\{x^{(1)}<z^{(1)}\}}g(z)\frac{(z^{(1)}-x^{(1)})^{\alpha/2}}{|x-z|^d}  {\rm d}z.
\label{otherhand}
\end{align}
where we have used \eqref{stablesubpot} and, in the last equality, we used L'H\^ospital rule.
Equating \eqref{onehand} with \eqref{otherhand} and noting that 
\[
2^\alpha E_{\alpha,d}\Gamma(\alpha/2) = \frac{\Gamma(d/2)}{ \pi^{d/2}\Gamma(\alpha/2)}
\]
completes the proof.
\end{proof}

\begin{remark}\rm We can remove the assumption that $g$ is compactly supported in the previous Lemma, leaving only the requirement that $g$ is continuous and bounded by appealing to a sequence of increasing functions $g_n\uparrow g$ and monotone convergence across the equality \eqref{excursionduality} applied to the sequence $g_n$.
\end{remark}
\begin{proof}[Proof of Theorem \ref{thm:triplelaw}]
Using the fact that   $X^{(1)}$ does not creep downwards, from the underlying L\'evy system  we get that 
\begin{align}
&\mathbb{E}_{w}\left[f(X_{G(\tau_r^{\vee})})g(X_{\tau_r^{\vee}-})h(X_{\tau_r^{\vee}}) \right] = \mathbb{E}_w\left[\int_0^{\tau_r^{\vee}}f(X_{G(t)})g(X_t)k(X_t){\rm d}t \right],
\label{levysystem}
\end{align}
where $k(y)=\int_{\{y^{(1)}+u^{(1)}<r\}}\Pi({\rm d}u)h(y+u)$.\\

Putting together identity \eqref{levysystem}  with Lemma \ref{lem:silverstein}, and rearranging the integrals we obtain

\begin{align}
&\mathbb{E}_w\left[\int_0^{\tau_r^{\vee}}f(X_{G(t)})g(X_t)k(X_t){\rm d}t \right]\notag\\
    &\quad =\int_{\{r<x^{(1)}<w^{(1)}\}}f(x)U_{w}^-({\rm d}x)  \mathbb{N}_{x^{(2:d)}}^-\left( \int_0^{\infty} (g\cdot k)(x^{(1)}+\epsilon^{(1)}(s),\epsilon^{(2:d)}(s))\dd s\right)\notag\\
    &\quad =\int_{\{r<x^{(1)}<w^{(1)}\}}f(x)U_{w}^-({\rm d}x)  \int _{\{x^{(1)}<y^{(1)}\}}g(y) U_x^+({\rm d}y)\int_{\{y^{(1)}+u^{(1)}<r\}}\Pi({\rm d}u)h(y+u)\label{Tanaka}\\
    &  \quad = \int_{\{r<x^{(1)}<w^{(1)}\}}\int_{\{x^{(1)}<y^{(1)}\}}\int_{\{z^{(1)}<r\}}f(x)g(y)h(z)U_{w}^-({\rm d}x)U_{x}^+({\rm d}y)K_{\alpha,d}|z-y|^{-(\alpha+d)}. \notag
\end{align}
where $K_{\alpha,d}=2^{\alpha}\pi^{-d/2}\Gamma((d+\alpha)/2)/|\Gamma(-\alpha/2)|$ and we have used \eqref{Pi}, from which the result follows.
\end{proof}

\begin{remark}\rm
The equality given by \eqref{Tanaka} agrees with a general formula given by \cite{tamura2008formula}, albeit in the latter, the factors where not clearly introduced as the potential measures of the ascending and descending ladder processes. 
\end{remark}
\section{Concluding remarks}

Many of the arguments we have given here are adapted from \cite{DEEP3}, which took a perspective on the stable process via the Lamperti--Kiu transform. Many of the calculations in \cite{DEEP3} demonstrate that the nature of the underlying  Markov Additive Process (MAP) captures the behaviour of the stable process via a generalised spherical polar coordinate system. The reason why we have been able to secure similarly explicit identities in relation to the hyperplane in this paper is because we have decomposed the stable process in an orthogonal coordinate system, that is, $X = (X^{(1)}, X^{(2:d)})$.
Moreover,  the pair $(X^{(1)}, X^{(2:d)})$ may  also be considered as a MAP and, although there is no time change involved as with the Lamperti--Kiu transform, the underlying principles of MAP fluctuation theory is still what drives our proofs to conclusion here, just as in \cite{DEEP3}. 
\smallskip

The assimilation of our `orthogonal' path decomposition in this article with what we know for `spherical polar' path decompositions of the stable process goes further than what we have have demonstrated in the previous sections. Below are but two examples of what more can be done. We leave them as exercises to the reader, with full proofs available in \cite{PhD}.

\subsection*{Orthogonal Wiener--Hopf factorisation of the stable process}

The Wiener--Hopf factorisation for one-dimensional L\'evy processes has undisputedly played a crucial role in understanding its fluctuations. The factorisation offers an actual decomposition of the L\'evy--Khintchine formula into two factors which correspond to the exponents of the ascending and descending ladder processes. In essence the factorisation is the result of splitting the path of the L\'evy process over an independent exponentially distributed time at its minimum  (or maximum) and applying duality to the post minimum (or  maximum). 

\smallskip

In higher dimensions it is less clear what such a factorisation should look like, however with an appropriate coordinate system, such as spherical polar or orthogonal coordinates one may still seek to split the path at a minimum (or maximum) in one of the coordinate directions.

Let us introduce the following potential operators, for $f$ a continuous function with compact support on $\mathbb{R}^{d-1}$, $\theta \in \mathbb{R}$, and $x^{(2:d)} \in \mathbb{R}^{d-1}$
$$\mathbf{R}_{\theta}[f](x):=\mathbb{E}_{0,x^{(2:d)}}\left[\int_0^{\infty}\ee ^{{\rm i} \theta X_t^{(1)}}f(X_t^{(2:d)}){\rm d}t \right].$$ 

Now define the potentials of the descending and ascending ladder processes, $(H^{(1)-},H^{(2:d)-})$ and  $(H^{(1)+},H^{(2:d)+})$ respectively (see the discussion at the start of Section \ref{prf1sect} for the latter), by
$$\hat{\boldsymbol{\rho}}_{\theta}[f](x):=\mathbb{E}_{0,x^{(2:d)}}\left[\int_0^{\infty}\ee^{{\rm i} \theta H_t^{(1)-}}f(H_t^{(2:d)-}){\rm d}t \right],$$
and
$$\boldsymbol {\rho}_{\theta}[f](x):=\mathbb{E}_{0,x^{(2:d)}}\left[\int_0^{\infty}\ee^{{\rm i} \theta H_t^{(1)+}}f(H_t^{(2:d)+}){\rm d}t \right].$$

The density of the potential of a $d$-dimensional stable process is a well known object (cf. Chapter 3 of \cite{KJCbook}), and allows us to write the operator $\mathbf{R}_{\theta}[f](x)$ as follows,

$$\mathbf{R}_{\theta}[f](x)=2^{-\alpha}\pi^{-d/2}\frac{\Gamma((d-\alpha)/2)}{\Gamma(\alpha/2)}\int_{\mathbb{R}^d}\ee^{{\rm i} \theta y^{(1)}}f(y^{(2:d)})|x-y|^{\alpha-d}{\rm d}y.$$

\begin{theorem}[Orthogonal  Wiener-Hopf  factorisation]
For a continuous compactly supported function $f:\mathbb{R}^{d-1}\to \mathbb{R}$, $\theta\in \mathbb{R}$ and $x=(0,x^{(2:d)}), \, x^{(2:d)} \in \mathbb{R}^{d-1}$ the potential operator $\mathbf{R}$ corresponding to the isotropic $\alpha$-stable  can be represented as the composition of ascending and descending potential operators as follows

\begin{equation}\label{deepfact} 
\mathbf{R}_{\theta}[f](x)=\hat{\boldsymbol{\rho}}_{\theta}\left[\boldsymbol{\rho}_{\theta}[f]\right](x).
\end{equation} 
Moreover, we have
$$\boldsymbol{\rho}_{\theta}[f](x)=\int_{\{y^{(1)}>0\}}\ee^{{\rm i} \theta y^{(1)}}f(y^{(2:d)})\frac{\Gamma(d/2)}{\pi^{d/2}\Gamma(\alpha/2)}\frac{(y^{(1)})^{\alpha/2}}{|x-y|^d}\dd y$$ and
$$\hat{\boldsymbol{\rho}}_{\theta}[f](x)=\boldsymbol{\rho}_{-\theta}[f](x).$$
\end{theorem}

\subsection*{Conditioning to hit plane continuously from one side}

The following results are an extension of those in \cite{tsogi}, where the law of the point of closest reach to the origin  is used to condition the isotropic $\alpha$-stable L\'evy process $(X,\mathbb{P})$ to continuously approach a patch  on the surface of the unit sphere from outside of the sphere. In the current setting, we can use the law of $X_{G(\tau_0^{\vee})}$ to condition $(X,\mathbb{P})$ to continuously approach a patch on the hyperplane $\mathbb{H}_0=\{x\in \mathbb{R}^d:x^{(1)}=0\}$ from one side.

\smallskip

Let us consider a Borel set $S\subseteq \mathbb{H}_0$ (which may be the whole of $\mathbb{H}_0$) such that it has strictly positive Lebesgue surface measure, written $\ell_{d-1}$, on $\mathbb{H}_0$ or it is a point. We aim to make sense of the law of $X$ conditioned to approach $S$ continuously from the half-space $\mathbb{H}_{0\uparrow}:=\{x\in \mathbb{R}^d:x^{(1)}>0\}$. Assume momentarily that  $S$ is not a point, but has strictly positive Lebesgue measure. Define for $\epsilon>0$ $S_{\epsilon}=(0,\epsilon)\times S$. Consider the event $C_\epsilon^{\vee}=\{X_{G(\tau_0^{\vee})}\in S_{\epsilon}\}$.
We are interested in the asymptotic conditioning
\begin{equation}
    \mathbb{P}_x^{\vee}(A,t<\zeta):=\lim_{\beta \to 0}\lim_{\epsilon \to 0}\mathbb{P}_x(A,t<\tau_{\beta}^{\vee}|C_{\epsilon}^{\vee}), \quad x\in \mathbb{H}_{0\uparrow},
\end{equation}
where $A\in \mathcal{F}_t = \sigma(X_u, u\leq t)$ and $\zeta$ is the lifetime of the process on $\mathbb{D}(\mathbb{R}^+\times \mathbb{R}^d)$, which we recall is the space of càdlàg paths from $\mathbb{R}^+$ to $\mathbb{R}^d$ with appended cemetery state, which is where the path is sent at its lifetime.
If $S=\{v\}\in \mathbb{H}_{0}$, the sets $S_{\epsilon}$ is adapted to converge to $S$ on $\mathbb{H}_{0\uparrow}$. For instance one could take $S_{\epsilon}=(0,\epsilon)\times \mathbb{B}^{(d-1)}(v,\epsilon)$ where $ \mathbb{B}^{(d-1)}(v,\epsilon)= \{y^{(2:d)}\in \mathbb{H}_0: |y^{(2:d)}-v|<\epsilon\}$. 
\smallskip

In order to state our main result let us define the following functions for $x\notin \mathbb{H}_0$,

\begin{equation} \label{hfunction}
  H_S(x)=\begin{cases} 
    \int_{ S}\big |x^{(1)}\big|^{\alpha/2}\bigg|\left(x^{(1)}\right)^2+|x^{(2:d)}-y^{(2:d)}|^2\bigg|^{-d/2} \ell_{d-1}(\dd y^{(2:d)}) & \text{ if } \ell_{d-1}(S) > 0,\\
   \big|x^{(1)}\big|^{\alpha/2}\bigg| \left(x^{(1)}\right)^2+|x^{(2:d)}-v^{(2:d)}|^2\bigg|^{-d/2} & \text{ if } S=\{v\}.
\end{cases}
\end{equation}
Note that $H_S$ is well defined when $S=\mathbb{H}_0$.

\begin{theorem}\label{thm:conditiononesided} Let $S\subseteq \mathbb{H}_0$ such that $\ell_{d-1}(S)>0$ or $S=\{v\}$ a point on $\mathbb{H}_0$. Then for all points of issue $x\notin \mathbb{H}_{0}$ we have as a change of measure on $\mathbb{D}(\mathbb{R}^+\times \mathbb{R}^d)$ such that, for all $A\in\mathcal{F}_t$,
\begin{equation}\label{radon1}
\mathbb{P}^\vee_x(A, \, t<\zeta) = \mathbb{E}_x\left[\mathbf{1}_{(A, \, t<\tau_0^{\vee})}\frac{H_S(X_t)}{H_S(x)}\right], \qquad  x\in \mathbb{H}_{0\uparrow},
\end{equation}
where  $H_S(x)$ is given in \eqref{hfunction}. 

\end{theorem}

\section*{Acknowledgements}
SM is supported by a scholarship from the EPSRC Centre for Doctoral Training in Statistical Applied Mathematics at Bath (SAMBa), under the project EP/S022945/1.

\bibliography{bibliography}{}
\bibliographystyle{abbrv}

\end{document}